\documentclass[11pt]{amsart}
\usepackage{amsmath,amssymb}
\usepackage{amsthm}
\usepackage{amsfonts}
\usepackage{mathrsfs}
\usepackage{needspace}
\usepackage{enumitem}%
\usepackage[margin=1in]{geometry}
\usepackage[english]{babel}
\usepackage[colorlinks, linkcolor = blue, anchorcolor = blue, citecolor = blue]{hyperref}

\numberwithin{equation}{section}

\newtheorem{mainthm}{Theorem}

\newtheorem{theorem}{Theorem}[section]
\newtheorem{definition}{Definition}[section]
\newtheorem{proposition}[theorem]{Proposition}
\newtheorem{corollary}[theorem]{Corollary}
\newtheorem{lemma}[theorem]{Lemma}
\newtheorem{remark}{Remark}[section]

\newtheorem{example}[theorem]{Example}

\newcommand{\R}{\mathbb R}
\newcommand{\Sym}{\operatorname{Sym}}
\newcommand{\tr}{\operatorname{tr}}
\newcommand{\Id}{\operatorname{Id}}
\newcommand{\Int}{\operatorname{Int}}

\newcommand{\ip}[2]{\langle #1,#2\rangle}
\newcommand{\op}{\mathrm{op}}
\newcommand{\cC}{\mathcal C}
\newcommand{\cF}{\mathcal F}
\newcommand{\Q}{\mathcal Q}
\newcommand{\bH}{\vec H}

\makeatletter
\@namedef{subjclassname@2020}{\textup{2020} Mathematics Subject Classification}
\makeatother

\begin{document}

\title[Tensorial pinching for mean curvature flow]{Tensorial pinching by deforming curvature cones in higher codimension mean curvature flow}
\author{Tsz-Kiu Aaron Chow}
\address{Department of Mathematics, Hong Kong University of Science and Technology, Hong Kong S.A.R., China}
\email{\href{chowtka@ust.hk}{chowtka@ust.hk}}
\thanks{T.-K. A. C. is supported  by the Croucher Foundation Start-up Grant and the HKUST New Faculty Start-up Grant.}

\begin{abstract}
We develop a pinching method for mean curvature flow in higher codimension. Motivated by pinching constructions for Ricci flow, the method uses a nonlinear deformation of the second fundamental form. We introduce two tensorial curvature conditions for closed immersions in Euclidean space which extend beyond the scope of classical scalar pinching results. In dimensions at least ten, we apply the method to prove convergence to a round point under the first condition and a canonical neighborhood theorem under the second.
\end{abstract}

\maketitle
\tableofcontents

\section{Introduction}

A basic question in mean curvature flow is how preserved curvature conditions restrict the geometry of singularities. In his seminal work on mean curvature flow, Huisken \cite{HuiskenConvex} proved that every closed strictly convex hypersurface contracts to a round point. For hypersurfaces, curvature conditions such as mean-convexity, convexity, and two-convexity place strong restrictions on possible singularity models \cite{HuiskenConvex,HuiskenSinestrariConvexity,HuiskenSinestrariSurgery}. In higher codimension, however, the second fundamental form has a more complicated algebraic structure, and relatively few curvature conditions are known to be preserved by the flow.

Andrews and Baker \cite{AB} extended Huisken's theorem to higher codimension under quadratic pinching. For $n\ge4$, their theorem assumes $H>0$ and \mbox{$|A|^2\le\frac1{n-1}H^2$}. This coefficient is sharp among scalar pinching conditions. Indeed, products $S^{n-1}(r)\times S^1(R)$ with $r/R$ small have $|A|^2/H^2$ arbitrarily close to $1/(n-1)$, but collapse onto a circle.

The broader scalar condition \mbox{$|A|^2\le\frac4{3n}H^2$} is also preserved by mean curvature flow \cite{AB}, but allows singularities other than round points. Under strict scalar pinching, Naff \cite{NaffPlanarity} proved a planarity estimate showing that the component of $A$ orthogonal to $\bH$ becomes negligible relative to $H$ at high curvature. Lynch \cite{LynchPlanarity} extended this estimate to the full strict Andrews--Baker range in the remaining dimensions five and six. Under suitable scalar pinching assumptions, convexity estimates were obtained by Lynch and Nguyen \cite{LynchNguyenConvexity}, and noncollapsing results by Naff \cite{NaffModels} and Bourni, Langford, and Lynch \cite{BourniLangfordLynch}. Naff \cite{NaffCanonical} established a canonical neighborhood theorem, and Lynch and Nguyen \cite{LynchNguyenSurgery} constructed mean curvature flow with surgery.

For ancient solutions, Lynch and Nguyen \cite{LynchNguyenAncient} proved rigidity under uniform quadratic pinching in the compact case. Lee, Naff, and Zhu \cite{LeeNaffZhu} proved a scale-invariant planarity estimate and obtained convexity and rigidity results for pinched ancient solutions. The higher-codimension planarity and surgery results above require scalar pinching at least as strong as \mbox{$|A|^2\le\frac4{3n}H^2$}.

Other approaches to pinching and codimension reduction appear in the work of Baker and Nguyen \cite{BakerNguyenNormal}, Colding and Minicozzi \cite{ColdingMinicozziComplexity}, and Langford \cite{LangfordPinching}. For related results in curved ambient spaces, see \cite{LiuXuYeZhaoHyperbolic,BakerNguyenSphere,PipoliSinestrari,LiuXuZhaoRiemannian,NguyenVogiatzi,VogiatziCP,LangfordLynchNguyenSphere}.

In this paper, we consider closed immersions $F_0:M^n\to\R^{n+q}$. We seek curvature conditions which allow $|A|^2>\frac4{3n}H^2$ and still imply round extinction or canonical neighborhoods. The following examples, whose calculations are given in Appendix~\ref{app:examples}, illustrate the limitations of scalar pinching.

\begin{example}\label{ex:intro-sphere}
Consider the Veronese embedding $J:\mathbb{RP}^{n+1}\to\Sym_0(\R^{n+2})$ given by $J([y])=yy^T-\Id/(n+2)$ for $|y|=1$, where $\Sym_0$ denotes the traceless symmetric matrices. We identify this space with $\R^N$, with Euclidean norm $|U|^2=\tr(U^TU)$. We consider the restriction of $J$ to a geodesic sphere of radius $r$ in projective space of sectional curvature one. The linear isometry from $\R^{n+1}\oplus\Sym_0(\R^{n+1})$ to $\Sym_0(\R^{n+2})$ is given by
\[
 (v,U)\longmapsto\begin{pmatrix}0&v^T/\sqrt2\\v/\sqrt2&U\end{pmatrix}.
\]
Using this, we identify $\R^{n+1}\oplus\Sym_0(\R^{n+1})$ with $\R^N$ using the product Euclidean norm and write the parametrization as
\[
 F_0(x)=\left(\sqrt2\cos r\sin r\,x,\sin^2r\left(xx^T-\frac{\Id}{n+1}\right)\right),\qquad x\in S^n.
\]
Choosing $\tan^2r=1/(2n)$, we obtain, for $n\ge10$,
\[
 \frac{|A|^2}{H^2}=\frac{3(n+1)}{2(n^2+n+1)}>\frac4{3n}>\frac1{n-1}.
\]
Nevertheless, its mean curvature flow contracts to a round point with multiplicity one. The evolution can be computed explicitly, and its second fundamental form satisfies the strict cone condition $\cC_1$ defined below.
\end{example}

\begin{example}\label{ex:intro-tube}
Apply the same Veronese construction to a geodesic $9$-sphere in $\mathbb{RP}^{10}$ with $\tan^2r=1/35$, dilate its image by $3\sqrt2$, and take its product with the circle $S^1(2)$. After a translation, we identify $\R^{10}\oplus\Sym_0(\R^{10})\oplus\R^2$ with $\R^N$ using the product Euclidean norm and write the parametrization as
\[
 F_0(x,z)=\left(\frac{\sqrt{35}}6x,\frac1{6\sqrt2}\left(xx^T-\frac1{10}\Id\right),2z\right),\qquad (x,z)\in S^9\times S^1.
\]
Its second fundamental form satisfies the strict cone condition $\cC_2$ defined below, whereas
\[
 \frac{|A|^2}{H^2}=\frac17>\frac4{3\cdot10}>\frac1{10-2}.
\]
Its mean curvature flow can also be computed explicitly. At the first singular time, the sphere factor contracts to a point while the circle has positive limiting radius. The tangent flows at every point of the limiting circle are shrinking round cylinders $S^9\times\R$, with multiplicity one.
\end{example}

These examples lead us to ask whether one can find preserved conditions on the full second fundamental form which include a larger class of immersions beyond scalar quadratic pinching and still imply round extinction or canonical neighborhoods. We investigate this direction by introducing two tensorial conditions on the second fundamental form.

Let $A_\alpha$ be the shape operators in an orthonormal normal frame, and put
\[
 A_\alpha^\circ=A_\alpha-\frac{\tr A_\alpha}{n}\Id,\qquad |A^\circ|^2=\sum_\alpha|A_\alpha^\circ|^2.
\]
We define two curvature cones
\begin{equation}\label{eq:cone1}
 \cC_1=\left\{A\in\Sym^2(\R^n)\otimes\R^q\;\middle|\;|A^\circ|^2\Id_{\R^n}+n\sum_\alpha(A_\alpha^\circ)^2\le\frac1{n-1}H^2\Id_{\R^n}\right\},
\end{equation}
\begin{equation}\label{eq:cone2}
 \cC_2=\left\{A\in\Sym^2(\R^n)\otimes\R^q\;\middle|\;
 \begin{aligned}
 &\frac{2(2n-1)}{n-1}|A^\circ|^2\Id_{\Lambda^2\R^n}+(n-4)\sum_\alpha((A_\alpha^\circ)^2)^{[2]}\\
 &\quad+2\sum_\alpha((A_\alpha^\circ)^{[2]})^2\le\frac{2(n+1)}{(n-1)(n-2)}H^2\Id_{\Lambda^2\R^n}
 \end{aligned}\right\}.
\end{equation}
Here, for an endomorphism $U$ of $\R^n$, we define $U^{[2]}$ on $\Lambda^2\R^n$ by $U^{[2]}(v\wedge w)=Uv\wedge w+v\wedge Uw$. These cones are closed and invariant under the natural $O(n)\times O(q)$ action, and
\[
 \left\{|A|^2\le\frac1{n-1}H^2\right\}\varsubsetneq\cC_1,\qquad \left\{|A|^2\le\frac1{n-2}H^2\right\}\varsubsetneq\cC_2.
\]
The first inclusion follows from the inequality $\sum_\alpha(A_\alpha^\circ)^2\le\frac{n-1}{n}|A^\circ|^2\Id_{\R^n}$, whereas the second follows from the inequality $(n-4)\sum_\alpha((A_\alpha^\circ)^2)^{[2]}+2\sum_\alpha((A_\alpha^\circ)^{[2]})^2\le(n-2)|A^\circ|^2\Id_{\Lambda^2\R^n}$.

Examples~\ref{ex:intro-sphere} and~\ref{ex:intro-tube} lie in the interiors of the respective cones and outside the scalar pinching regions (see Appendix~\ref{app:examples}). The tensors of $S^{n-1}\times\R$ and $S^{n-2}\times\R^2$ lie on the boundaries of $\cC_1$ and $\cC_2$, respectively. Thus these conditions retain the relevant cylinders while allowing round extinction and cylindrical singularities beyond the scalar bound $4/(3n)$. To that end, we prove

\begin{mainthm}\label{thm:mainA}
Let $M^n$ be closed and connected, with $n\ge10$, and let $F:M^n\times[0,T)\to\R^{n+q}$ be a maximal smooth solution of mean curvature flow of immersions such that $A(p,0)\in\cC_1$ for every $p\in M$. Then
\begin{enumerate}[label=(\roman*)]
\item $A(p,t)\in\cC_1$ for all $(p,t)\in M\times[0,T)$.
\item $A(p,t)\in\Int\cC_1$ and $H(p,t)>0$ for all $(p,t)\in M\times(0,T)$.
\item There exists $x_*\in\R^{n+q}$ such that the rescaled flow $(T-t)^{-1/2}(F(\cdot,t)-x_*)$ converges smoothly, after reparametrization, to a round sphere as $t\uparrow T$. In particular, $M$ is diffeomorphic to $S^n$.
\end{enumerate}
\end{mainthm}

For the second theorem, an \emph{ancient model solution} is understood in the sense of \cite[Definition~2.2]{NaffCanonical}. We specify the class of models and the meaning of closeness in Section~\ref{sec:canonical}.

\begin{mainthm}\label{thm:mainB}
Let $M^n$ be closed and connected, with $n\ge10$, and let $F:M^n\times[0,T)\to\R^{n+q}$ be a maximal smooth solution of mean curvature flow of immersions such that $A(p,0)\in\cC_2$ for every $p\in M$. Write $F_0=F(\cdot,0)$. Then
\begin{enumerate}[label=(\roman*)]
\item $A(p,t)\in\cC_2$ for all $(p,t)\in M\times[0,T)$.
\item $A(p,t)\in\Int\cC_2$ and $H(p,t)>0$ for all $(p,t)\in M\times(0,T)$.
\item There exist constants $\widetilde\varepsilon>0$ and $\widetilde K>0$, depending on $F_0$, with the following property. For every $\varepsilon_0\in(0,\widetilde\varepsilon)$ and $K_0>\widetilde K$, there is a number $\widehat r=\widehat r(F_0,\varepsilon_0,K_0)>0$ such that, if $(p_0,t_0)\in M\times[0,T)$ satisfies $Q_0:=H(p_0,t_0)\ge\widehat r^{-1}$, then the flow is $\varepsilon_0$-close in the parabolic neighborhood $B_{g(t_0)}(p_0,Q_0^{-1}K_0)\times[t_0-K_0Q_0^{-2},t_0]$ to an ancient model solution, in the sense of Definition~\ref{def:comparison}.
\end{enumerate}
\end{mainthm}

The classical pinching estimates for mean curvature flow in higher codimension do not apply directly under \eqref{eq:cone1} and \eqref{eq:cone2}, since these cone conditions extend beyond the scalar pinching range. For these tensorial conditions, it is also not clear how to construct a scalar quantity satisfying a differential inequality from which the required pinching estimates follow by Stampacchia iteration or the maximum principle. We instead deform the curvature cones directly, drawing on the pinching constructions for Ricci flow.

B\"ohm and Wilking \cite{BW} introduced linear maps $\ell_{a,b}$ which change the scalar and traceless Ricci components of an algebraic curvature tensor while leaving its Weyl component unchanged. For suitable parameters, these maps produce families of preserved curvature cones and pinching sets. This construction underlies their convergence theorem, the sphere theorem of Brendle and Schoen \cite{BrendleSchoen}, and Brendle's convergence theorem \cite{BrendleConvergence,BrendleBook}. In these results, the pinching sets force normalized curvature tensors to approach constant positive curvature as the scalar curvature tends to infinity.

Pinching constructions can also retain cylindrical singularities, as in the work of Brendle \cite{BrendlePIC} and Brendle and Tsiamis \cite{BrendleTsiamis} on Ricci flow with positive isotropic curvature. These results suggest choosing the limiting curvature condition according to the singularities one expects to retain.

There is an essential difference between mean curvature flow and Ricci flow. For a hypersurface with $H>0$, the Hamilton-type reaction ODE associated with the evolution equation is $\frac{dA}{dt}=|A|^2A$, and hence $\frac d{dt}(A/H)=0$. Thus the reaction ODE does not improve the ratios of the principal curvatures to $H$. We therefore seek a deformation which makes the components orthogonal to $\bH$ small relative to $H$, while retaining the hypersurface part. As in \cite{NaffPlanarity}, for $H>0$ we write
\[
 A^{\bH}=\langle A,\bH\rangle_N,\qquad A=\frac{A^{\bH}}{H^2}\otimes\bH+\widehat A,
\]
and define
\[
 \Theta_s(A)=\widehat A+s\frac{A^{\bH}}{H^2}\otimes\bH,\qquad s\ge1.
\]
This map multiplies the component in the mean-curvature direction by $s$ and leaves $\widehat A$ unchanged. Since $\tr\Theta_s(A)=s\bH$, it reduces the ratio $|\widehat A|/H$ by a factor of $s$. The map is nonlinear, since the normal direction $\bH/H$ depends on $A$.

In Section~3, we prove that both cones are preserved and that, at any fixed positive time, $H>0$ and the second fundamental forms lie in a smaller preserved cone $\cC$ with $\cC\setminus\{0\}\subset\Int\cC_i$. We consider the deformed cones
\[
 \cC(s)=\Theta_s(\cC\cap\{H>0\}),\qquad s\ge1.
\]
Their defining inequalities imply $s^2|\widehat A|^2\le2H^2/n$. We adapt the pinching-set construction in \cite[Section~8.3]{BrendleBook} to obtain a preserved set containing the second fundamental forms at the chosen positive time, on which, for every $s\ge1$, there is a constant $h_s>0$ such that
\[
 A+\frac{h_s}{H}\Id\otimes\bH\in\cC(s).
\]
The tensor on the left has the same component $\widehat A$ and mean-curvature norm $H+nh_s$. Hence
\[
 \frac{|\widehat A|}{H}\le\frac1s\sqrt{\frac2n}\left(1+\frac{nh_s}{H}\right).
\]
For fixed $s$, the term $nh_s/H$ tends to zero as $H\to\infty$. Since $s$ can be arbitrarily large, $|\widehat A|/H$ tends to zero uniformly as $H\to\infty$. The construction is carried out in Sections~\ref{sec:deformation} and~\ref{sec:pinching}. Together with derivative estimates, this pinching implies that every connected nonflat smooth blow-up limit lies in a fixed affine $(n+1)$-plane.

The remaining arguments use the geometry of the initial cones. For $\cC_1$, the positive-time cone condition makes the limiting hypersurfaces uniformly convex. Hamilton's theorem for complete hypersurfaces with pinched second fundamental form \cite{HamiltonPinched} and the rigidity theorem for ancient convex solutions \cite{HuiskenSinestrariAncient} then give shrinking spheres. The maximum principle, Huisken's monotonicity formula, and uniqueness of compact tangent flows \cite{Schulze} yield the convergence in Theorem~\ref{thm:mainA}(iii).

For $\cC_2$, the reduction to hypersurfaces allows us to apply the maximum principle to $|A|^2/H^2$ and obtain the cylindrical estimate $\limsup_{H\to\infty}|A|^2/H^2\le1/(n-1)$. Consequently, the limiting hypersurfaces are weakly convex and uniformly two-convex. The derivative estimates and Naff's arguments for complete ancient limits \cite{NaffModels,NaffCanonical} then give the canonical neighborhoods in Theorem~\ref{thm:mainB}(iii).

\textit{Disclosure of AI use.} The curvature cones were constructed by the author. The ideas underlying the deformation in Section~4 and the pinching argument in Section~5 are due to the author and are motivated by pinching constructions for Ricci flow. AI tools (ChatGPT 6) were used to assist with algebraic calculations in Section~3, revising the exposition, literature searches, and exploration of other related ideas. The author takes full responsibility for the mathematical content and the final manuscript.

\section{Maximum principles and pinching sets}

In this section, we prove a maximum principle for curvature conditions defined by families of functions on $\Sym^2(V)\otimes N$. We then define pinching sets and give conditions under which nonflat blow-up limits lie in a fixed affine $(n+1)$-plane.

\subsection{Notation and evolution equations}
Let $V$ and $N$ be real vector spaces of dimensions $n$ and $q$, equipped with inner products, and identify $\Sym^2(V)$ with the self-adjoint endomorphisms of $V$. Fix an orthonormal basis $\nu_1,\ldots,\nu_q$ of $N$. For $A\in\Sym^2(V)\otimes N$ and $\xi\in N$, we denote $\ip A\xi_N(X,Y)=\ip{A(X,Y)}\xi$. Moreover, we write
\[
 A=\sum_{\alpha=1}^q A_\alpha\otimes\nu_\alpha,\qquad \bH=\tr A,\qquad H=|\bH|,\qquad A^{\bH}=\ip A{\bH}_N.
\]
We use the norm $|U|^2=\tr(U^TU)$ for endomorphisms of $V$ and define $\Q(A)\in\Sym^2(V)\otimes N$ by
\begin{equation}\label{eq:reaction}
 \Q(A)_\alpha=\sum_\beta\left( \tr(A_\alpha A_\beta)A_\beta +A_\beta^2A_\alpha+A_\alpha A_\beta^2 -2A_\beta A_\alpha A_\beta\right).
\end{equation}
The map $\Q$ commutes with the natural $O(V)\times O(N)$ action.

Let $F:M\times I\to\R^{n+q}$ be a smooth mean curvature flow, with induced metric $g$ and second fundamental form $A$. By \cite[(25)--(26) and Section~2.3]{AB}, we have
\begin{align}
 (\nabla_t-\Delta)A&=\Q(A),\label{eq:evolution}\\
 (\nabla_t^\perp-\Delta^\perp)\bH &=\sum_\alpha\tr(A_\alpha A^{\bH})\nu_\alpha,\label{eq:mean-evolution}
\end{align}
where the connections in time are defined by $\nabla_tX=\partial_tX-A^{\bH}X$ and $\nabla_t^\perp\xi=(\partial_t\xi)^\perp$.
The Codazzi equation gives $\nabla A\in\Sym^3(TM)\otimes NM$.

\subsection{Maximum principles}
Let $Y\subset\Sym^2(V)\otimes N$ be invariant under the natural $O(V)\times O(N)$ action. For a smooth mean curvature flow $F:M\times I\to\R^{n+q}$, choose at each point $(p,t)$ linear isometries from $V$ to $(T_pM,g)$ and from $N$ to the normal space $N_pM$ of $F(\cdot,t)$ at $p$. These induce a linear isometry from $\Sym^2(V)\otimes N$ to $\Sym^2(T_pM)\otimes N_pM$. Let $Y(p,t)$ be the image of $Y$. Since $Y$ is invariant under the natural $O(V)\times O(N)$ action, $Y(p,t)$ is independent of these choices. We write $A(p,t)\in Y$ to mean $A(p,t)\in Y(p,t)$.

Fix $1\le k\le n$, and let $W\subset\Lambda^kV$ be a nonempty compact set invariant under the natural $O(V)$ action, with $O(N)$ acting trivially. Let $\{\Phi_u\}_{u\in W}$ be a continuous family of smooth functions on an open neighborhood of $Y$ invariant under this action, such that
\[
 \Phi_{gu}(gA)=\Phi_u(A) \qquad\text{for all }g\in O(V)\times O(N).
\]
Define
\[
 \cC=\{A\in Y:\min_{u\in W}\Phi_u(A)\ge0\}.
\]
The equivariance of the functions $\Phi_u$ implies that $\cC$ is invariant under the natural $O(V)\times O(N)$ action, and we use the same convention for $A(p,t)\in\cC$.
\begin{proposition}\label{prop:maximum}
Suppose that the condition $A\in Y$ is preserved by mean curvature flow of closed immersions. Suppose also that there is a constant $c\ge0$ such that
\begin{equation}\label{eq:maximum-hypothesis}
 D\Phi_u(A)[\Q(A)]-\sum_jD^2\Phi_u(A)[Z_j,Z_j] \ge c|A|^2\Phi_u(A)
\end{equation}
whenever $A\in Y$, $\Phi_u(A)=\min_{v\in W}\Phi_v(A)\le0$, and $Z\in\Sym^3(V)\otimes N$, where $Z_j=Z(e_j,\cdot,\cdot)$. Then $\cC$ is preserved by mean curvature flow of closed immersions.
\end{proposition}

\begin{proof}
Let $F:M\times[0,T)\to\R^{n+q}$ be a smooth mean curvature flow of a closed immersion with $A(\cdot,0)\in\cC$. Fix $0<T'<T$ and $L>c\sup_{M\times[0,T']}|A|^2$. Since $A\in Y$ is preserved, the flow remains in $Y$. Fix $\epsilon>0$, and suppose that $\min_{v\in W}\Phi_v(A)+\epsilon e^{Lt}$ has its first zero at $(x_0,t_0)\in M\times(0,T']$. Then
\[
 \Phi_u(A(x_0,t_0))=\min_{v\in W}\Phi_v(A(x_0,t_0))=-\epsilon e^{Lt_0}<0
\]
for some $u\in W$.

Fix orthonormal tangent and normal frames at $(x_0,t_0)$. Extend them by parallel transport along the time line through $x_0$, and then along radial geodesics from $x_0$ at each time, keeping $u$ fixed in these frames. We have $\Phi_u(A)+\epsilon e^{Lt}\ge0$ near $(x_0,t_0)$ for $t\le t_0$, with equality at $(x_0,t_0)$. By \eqref{eq:evolution}, the Codazzi equation, and \eqref{eq:maximum-hypothesis}, we obtain at $(x_0,t_0)$
\begin{align*}
 0&\ge(\partial_t-\Delta)\bigl(\Phi_u(A)+\epsilon e^{Lt}\bigr)\\
 &=D\Phi_u(A)[\nabla_tA-\Delta A]-\sum_jD^2\Phi_u(A)[\nabla_jA,\nabla_jA]+L\epsilon e^{Lt_0}\\
 &=D\Phi_u(A)[\Q(A)]-\sum_jD^2\Phi_u(A)[\nabla_jA,\nabla_jA]+L\epsilon e^{Lt_0}\\
 &\ge c|A|^2\Phi_u(A)+L\epsilon e^{Lt_0}=(L-c|A|^2)\epsilon e^{Lt_0}>0,
\end{align*}
a contradiction. Thus $\min_{u\in W}\Phi_u(A)+\epsilon e^{Lt}>0$ on $M\times[0,T']$. Letting $\epsilon\downarrow0$, we obtain $A\in\cC$ on this interval. Since $T'<T$ was arbitrary, the assertion follows.
\end{proof}

\begin{remark}
The hypothesis is imposed also where $\min_{u\in W}\Phi_u(A)<0$, since the defining functions may vanish to higher order on the boundary of $\cC$. An inequality imposed only at their zeros may therefore fail to detect motion out of $\cC$. For example, $\Phi(A)=(1-|A|^2)^3$ satisfies $D\Phi=D^2\Phi=0$ when $|A|^2=1$, but the condition $|A|^2\le1$ is not preserved by mean curvature flow of round spheres. In contrast, Hamilton's maximum principle for Ricci flow uses a geometric tangency condition for closed convex sets, independent of the choice of defining functions.
\end{remark}

\begin{samepage}
\begin{proposition}\label{prop:maximum-equality}
Suppose that
\begin{equation}\label{eq:maximum-equality-hypothesis}
 D\Phi_u(A)[\Q(A)]-\sum_jD^2\Phi_u(A)[Z_j,Z_j]\ge0
\end{equation}
whenever $A\in\cC$, $\Phi_u(A)=\min_{v\in W}\Phi_v(A)$, and $Z\in\Sym^3(V)\otimes N$, where $Z_j=Z(e_j,\cdot,\cdot)$. Let a smooth mean curvature flow on a connected manifold satisfy $A\in\cC$ on $M\times(t_1,t_0]$. If $\min_{u\in W}\Phi_u(A(p,t_0))=0$ at some point $p\in M$, then $\min_{u\in W}\Phi_u(A)=0$ on $M\times(t_1,t_0]$.
\end{proposition}
\end{samepage}

\begin{proof}
To apply the strong minimum principle to $\min_{v\in W}\Phi_v(A)$, let $\psi$ be a smooth function such that $\psi\le\min_{v\in W}\Phi_v(A)$ near a point $(x,s)\in M\times(t_1,t_0]$ for $t\le s$, with equality at $(x,s)$. Then
\[
 \psi(x,s)=\min_{v\in W}\Phi_v(A(x,s))=\Phi_u(A(x,s))
\]
for some $u\in W$. Fix orthonormal tangent and normal frames at $(x,s)$, extend them by parallel transport along the time line through $x$ and along radial geodesics at each time, and keep $u$ fixed in these frames. Since $\psi\le\Phi_u(A)$ with equality at $(x,s)$, \eqref{eq:evolution}, the Codazzi equation, and \eqref{eq:maximum-equality-hypothesis} give
\begin{align*}
 (\partial_t-\Delta)\psi&\ge(\partial_t-\Delta)\Phi_u(A)\\
 &=D\Phi_u(A)[\Q(A)]-\sum_jD^2\Phi_u(A)[\nabla_jA,\nabla_jA]\ge0
\end{align*}
at $(x,s)$. Thus $\min_{u\in W}\Phi_u(A)$ is a nonnegative viscosity supersolution of the heat equation.

For any $x\in M$, apply the strong minimum principle \cite[Corollary~2.4]{DaLio} on a relatively compact connected domain containing $p$ and $x$. Since $\min_{u\in W}\Phi_u(A(p,t_0))=0$, it follows that $\min_{u\in W}\Phi_u(A(x,t))=0$ for $t_1<t<t_0$, and also for $t=t_0$ by continuity.
\end{proof}

\subsection{Pinching sets and reduction of codimension}
We use the following analogue of \cite[Definition~5.12]{BrendleBook}.

\begin{definition}\label{def:pinching}
Let $\cF\subset\{A\in\Sym^2(V)\otimes N:H>0\}$ be nonempty, closed in $\{H>0\}$, and invariant under positive scalar multiplication and the natural $O(V)\times O(N)$ action. Suppose that every $A\in\cF$ satisfies
\[
 \dim\operatorname{span}\{A(X,Y):X,Y\in V\}=1, \qquad \operatorname{rank}A^{\bH}\ge2.
\]
A set $\mathcal P\subset\{A\in\Sym^2(V)\otimes N:H>0\}$ is called $\cF$-pinching if the following conditions hold:
\begin{itemize}
\item $\mathcal P$ is invariant under the natural $O(V)\times O(N)$ action.
\item $A\in\mathcal P$ is preserved by mean curvature flow of closed immersions in $\R^{n+q}$.
\item For every $\rho>0$, the set $\{A\in\mathcal P:\operatorname{dist}(A,\cF)>\rho|A|\}$ is bounded.
\end{itemize}
\end{definition}

Throughout this subsection and the next, $\cF$ satisfies the assumptions in Definition~\ref{def:pinching}.

\begin{lemma}\label{lem:plane}
Let $F:M\times I\to\R^{n+q}$ be a smooth mean curvature flow, where $M$ is connected and $I$ is an interval. Suppose that $A\in\cF$. Then the image of the flow lies in a fixed affine $(n+1)$-plane.
\end{lemma}

\begin{proof}
Let $\nu=\bH/H$ and $h=A^{\bH}/H$. Then $A=h\otimes\nu$ and $\operatorname{rank}h\ge2$. At any point, choose an orthonormal eigenbasis of $h$. Since $h$ has at least two nonzero eigenvalues, for each $i$ there is $j\ne i$ with $h_{jj}\ne0$. The component orthogonal to $\nu$ in the Codazzi equation gives
\[
 h_{jj}\nabla_i^\perp\nu=h_{ij}\nabla_j^\perp\nu=0.
\]
Thus $\nabla^\perp\nu=0$. Equation~\eqref{eq:mean-evolution} now becomes
\[
 ((\partial_t-\Delta)H)\nu+H\nabla_t^\perp\nu=H|h|^2\nu,
\]
and hence $\nabla_t^\perp\nu=0$.

Let $\overline\nabla$ denote the Euclidean connection. We have
\[
 \overline\nabla_X\mathrm dF(Y)=\mathrm dF(\nabla_XY)+h(X,Y)\nu,\qquad \overline\nabla_X\nu=-\mathrm dF(hX).
\]
For a vector field $X$ on $M$, $\partial_tF=H\nu$ and $\langle\nu,\mathrm dF(X)\rangle=0$ give
\[
 \partial_t\mathrm dF(X)=X(H)\nu-H\mathrm dF(hX),\qquad \partial_t\nu=-\mathrm dF(\nabla H).
\]
These identities show that $\mathrm dF(TM)\oplus\R\nu$ is a fixed $(n+1)$-dimensional linear subspace $P$. Let $\pi_{P^\perp}$ denote orthogonal projection onto $P^\perp$. Then
\[
 \mathrm d(\pi_{P^\perp}F)=0,\qquad \partial_t(\pi_{P^\perp}F)=H\pi_{P^\perp}\nu=0.
\]
Since $M\times I$ is connected, $\pi_{P^\perp}F$ is constant. Thus the image of the flow lies in a fixed affine $(n+1)$-plane.
\end{proof}

The following lemma will ensure that $H$ is positive on the nonflat limits to which we apply Lemma~\ref{lem:plane}.

\begin{lemma}\label{lem:complete-positive}
Let $F:M\times I\to\R^{n+q}$ be a smooth mean curvature flow, where $M$ is connected and $I$ is an interval. Suppose that $H$ is positive at some point and that, for a constant $C<\infty$,
\[
 |\nabla H^{-1}|+|\partial_tH^{-2}|\le C \qquad\text{wherever }H>0.
\]
Then $H>0$ throughout $M\times I$.
\end{lemma}

\begin{proof}
Choose $(p,t)\in M\times I$ with $H(p,t)>0$. Suppose that $H(p,s)\le0$ for some $s\in I$ with $s>t$, and set $s_0=\inf\{s\in I:s>t,\ H(p,s)\le0\}$. For $t<s<s_0$, integrating $|\partial_sH^{-2}|\le C$ gives
\[
 H(p,s)^{-2}\le H(p,t)^{-2}+C|s-t|\le H(p,t)^{-2}+C|s_0-t|.
\]
The right-hand side is bounded, giving a contradiction as $s\uparrow s_0$. The same argument for $s<t$ shows that $H(p,s)>0$ for every $s\in I$.

Fix $(x,s)\in M\times I$, and choose a smooth path $\gamma:[0,1]\to M$ from $p$ to $x$. Suppose that $H(\gamma(r),s)$ vanishes, and let $r_0\in(0,1]$ be its first zero. For $0\le r<r_0$, integration of $|\nabla H^{-1}|\le C$ along $\gamma$ gives
\[
 H(\gamma(r),s)^{-1}\le H(p,s)^{-1}+C\int_0^r|\dot\gamma(\rho)|_{g(s)}\,d\rho\le H(p,s)^{-1}+C\operatorname{Length}_{g(s)}(\gamma).
\]
The right-hand side is bounded, giving a contradiction.
\end{proof}

\subsection{Blow-up limits}
For a flow $F:M\times[0,T)\to\R^{n+q}$, points $(p_j,t_j)$, and $\lambda_j>0$, define
\begin{equation}\label{eq:rescaling}
 F_j(x,s)=\lambda_j \bigl(F(x,t_j+\lambda_j^{-2}s)-F(p_j,t_j)\bigr),\qquad -\lambda_j^2t_j\le s<\lambda_j^2(T-t_j).
\end{equation}
Let $A_j$ denote the second fundamental form of $F_j$. We use smooth pointed convergence in the sense of \cite[Definition~4.3]{NaffCanonical}.

In Sections~6 and~7, the bound in Lemma~\ref{lem:complete-positive} will follow from Proposition~\ref{prop:derivatives}.

\begin{proposition}\label{prop:blowup}
Let $\mathcal P$ be $\cF$-pinching, and let a closed connected smooth flow satisfy $A(\cdot,t_0)\in\mathcal P$. Assume that $|A|\le CH$ for a constant $C<\infty$ and that the derivative bound in Lemma~\ref{lem:complete-positive} holds on $M\times[t_0,T)$. Suppose $\lambda_j\to\infty$ and $t_j\ge t_0$. Every connected nonflat smooth pointed limit of \eqref{eq:rescaling} satisfies $H_\infty>0$ and $A_\infty\in\cF$ throughout, and lies in a fixed affine $(n+1)$-plane.
\end{proposition}

\begin{proof}
Let $F_\infty$ be a connected nonflat smooth pointed limit of \eqref{eq:rescaling}. If $\lambda_j^2(t_j-t_0)$ were bounded along a subsequence, then $t_j\to t_0$ and the curvature bound near $t_0$ would give $A_\infty=0$. Thus $\lambda_j^2(t_j-t_0)\to\infty$. By \eqref{eq:rescaling}, the estimates on $M\times[t_0,T)$ pass to the limit, so $|A_\infty|\le CH_\infty$ and the bound in Lemma~\ref{lem:complete-positive} holds wherever $H_\infty>0$. Since the limit is nonflat, $H_\infty$ is positive somewhere, and Lemma~\ref{lem:complete-positive} gives $H_\infty>0$ everywhere.

Fix $\rho>0$. Since $A\in\mathcal P$ is preserved, Definition~\ref{def:pinching} gives a constant $R_\rho$ such that $\operatorname{dist}(A,\cF)\le\rho|A|+R_\rho$ for $t\ge t_0$. The invariance of $\cF$ under positive scalar multiplication then gives
\[
 \operatorname{dist}(A_j,\cF) =\lambda_j^{-1}\operatorname{dist}(A,\cF) \le\rho|A_j|+\lambda_j^{-1}R_\rho.
\]
Letting $j\to\infty$ and then $\rho\downarrow0$, we obtain $\operatorname{dist}(A_\infty,\cF)=0$. Since $H_\infty>0$ and $\cF$ is closed in $\{H>0\}$, we have $A_\infty\in\cF$. Lemma~\ref{lem:plane} now implies that the image lies in a fixed affine $(n+1)$-plane.
\end{proof}

\subsection{Hypersurface rigidity}
\begin{lemma}\label{lem:ratio}
Let $F:M\times(t_1,t_0]\to\R^{n+1}$ be a smooth mean curvature flow of a connected hypersurface with $H>0$. Suppose $|A|^2\le\beta H^2$ for some constant $\beta<1$. If $|A(p,t_0)|^2=\beta H(p,t_0)^2$ at some $p\in M$, then $\nabla A=0$ on $M\times(t_1,t_0]$, and all nonzero principal curvatures are equal and positive.
\end{lemma}
\begin{proof}
Set $f=|A|^2/H^2$. The hypersurface evolution equations give
\[
 (\partial_t-\Delta)f =\frac2H\ip{\nabla H}{\nabla f} -\frac2{H^4}|H\nabla A-A\otimes\nabla H|^2.
\]
The strong maximum principle \cite[Corollary~2.4]{DaLio} gives $f\equiv\beta$, since $f\le\beta$ and $f(p,t_0)=\beta$. Hence $H\nabla A=A\otimes\nabla H$, and Codazzi yields
\[
 A_{jk}\partial_iH=A_{ik}\partial_jH.
\]
Choose an orthonormal eigenbasis of $A$. Since $|A|^2<H^2$, at least two principal curvatures are nonzero. For each $i$, choose $j\ne i$ with $A_{jj}\ne0$ and take $k=j$ in the above identity. We obtain $\partial_iH=0$, and therefore $\nabla H=0$ and $\nabla A=0$. Commuting covariant derivatives of $A$ and using the Gauss equation gives
\[
 \lambda_j\lambda_k(\lambda_j-\lambda_k)=0 \qquad(j\ne k).
\]
All nonzero principal curvatures are therefore equal, and they are positive because $H>0$.
\end{proof}

\allowdisplaybreaks[2]
\raggedbottom



\section{Preserved curvature conditions}\label{sec:conditions}

We prove that the conditions $\cC_1$ and $\cC_2$ are preserved by
mean curvature flow. Throughout this section, $n\ge10$ and
$0\le\eta\le1/2$. Let $V$ and $N$ be the vector spaces fixed in Section~2. For $A\in\Sym^2(V)\otimes N$, write
\[
 A^\circ=A-\frac1n\Id\otimes\bH.
\] 
For a unit vector $u\in V$, define
\begin{equation}\label{eq:phi1}
 \Phi_u^{1,\eta}(A)=\frac{1-\eta}{n-1}H^2-|A^\circ|^2-n\sum_\alpha|A_\alpha^\circ u|^2.
\end{equation}
For a unit two-form $u\in\Lambda^2V$, define
\begin{equation}\label{eq:phi2}
\begin{aligned}
 \Phi_u^{2,\eta}(A)
 &=\frac{2(n+1)(1-\eta)}{(n-1)(n-2)}H^2-\frac{2(2n-1)}{n-1}|A^\circ|^2\\
 &\quad -(n-4)\sum_\alpha
             \ip{((A_\alpha^\circ)^2)^{[2]}u}u-2\sum_\alpha|(A_\alpha^\circ)^{[2]}u|^2.
\end{aligned}
\end{equation}
These two families of quadratic forms are equivariant under the natural $O(V)\times O(N)$ action. We define
\[
 \cC_i(\eta)=\{A:\Phi_u^{i,\eta}(A)\ge0
                    \text{ for every unit }u\},
\]
where $u\in V$ for $i=1$ and $u\in\Lambda^2V$ for $i=2$. Observe that by definition we have $\cC_i=\cC_i(0)$.

A direct calculation using \eqref{eq:reaction} gives the following
identities.

\Needspace{8\baselineskip}
\begin{lemma}\label{lem:phi-derivatives}
Let $A,Z,W\in\Sym^2(V)\otimes N$.
\begin{enumerate}[label=(\roman*)]
\item If $u\in V$ is a unit vector, then
\begin{equation}\label{eq:phi1-derivative}
 \frac12D\Phi_u^{1,\eta}(A)[\Q(A)] =\frac{1-\eta}{n-1}|A^{\bH}|^2 -\sum_\alpha\tr\bigl(A_\alpha^\circ(\Q(A)_\alpha)^\circ\bigr) -n\sum_\alpha\ip{A_\alpha^\circ u}{(\Q(A)_\alpha)^\circ u},
\end{equation}
\begin{equation}\label{eq:phi1-hessian}
 \frac12D^2\Phi_u^{1,\eta}(A)[Z,W] =\frac{1-\eta}{n-1}\ip{\tr Z}{\tr W} -\sum_\alpha\tr(Z_\alpha^\circ W_\alpha^\circ) -n\sum_\alpha\ip{Z_\alpha^\circ u}{W_\alpha^\circ u}.
\end{equation}

\item If $u\in\Lambda^2V$ is a unit two-form, then
\begin{equation}\label{eq:phi2-derivative}
\begin{aligned}
 \frac12D\Phi_u^{2,\eta}(A)[\Q(A)] &=\frac{2(n+1)(1-\eta)}{(n-1)(n-2)}|A^{\bH}|^2 -\frac{2(2n-1)}{n-1} \sum_\alpha\tr\bigl(A_\alpha^\circ(\Q(A)_\alpha)^\circ\bigr)\\
 &\quad-(n-4)\sum_{\alpha,j} \ip{A_\alpha^\circ\iota_{e_j}u} {(\Q(A)_\alpha)^\circ\iota_{e_j}u}\\
 &\quad-2\sum_\alpha \ip{(A_\alpha^\circ)^{[2]}u} {((\Q(A)_\alpha)^\circ)^{[2]}u},
\end{aligned}
\end{equation}
\begin{equation}\label{eq:phi2-hessian}
\begin{aligned}
 \frac12D^2\Phi_u^{2,\eta}(A)[Z,W] &=\frac{2(n+1)(1-\eta)}{(n-1)(n-2)}\ip{\tr Z}{\tr W} -\frac{2(2n-1)}{n-1}\sum_\alpha\tr(Z_\alpha^\circ W_\alpha^\circ)\\
 &\quad-(n-4)\sum_{\alpha,j} \ip{Z_\alpha^\circ\iota_{e_j}u}{W_\alpha^\circ\iota_{e_j}u} -2\sum_\alpha\ip{(Z_\alpha^\circ)^{[2]}u}{(W_\alpha^\circ)^{[2]}u}.
\end{aligned}
\end{equation}

\item For either family, we have
\begin{equation}\label{eq:phi-quadratic-identities}
 D\Phi_u^{i,\eta}(A)[Z]=D^2\Phi_u^{i,\eta}(A)[A,Z],\qquad \frac12D^2\Phi_u^{i,\eta}(A)[Z,Z]=\Phi_u^{i,\eta}(Z).
\end{equation}
Moreover, the symmetric bilinear form $(Z,W)\mapsto-D^2\Phi_u^{i,\eta}(A)[Z,W]$ defines an inner product on traceless tensors.
\end{enumerate}
\end{lemma}

\begin{lemma}\label{lem:scalar}
Let $A\in\Sym^2(V)\otimes N$. For a unit vector $u\in V$ and a unit two-form $u\in\Lambda^2V$, respectively, we have
\[
\begin{gathered}
 |A^\circ|^2\le-\Phi_u^{1,\eta}(A^\circ)\le n|A^\circ|^2,\\
 \frac{2(2n-1)}{n-1}|A^\circ|^2 \le-\Phi_u^{2,\eta}(A^\circ)\le\frac{n(n+1)}{n-1}|A^\circ|^2.
\end{gathered}
\]
If $A\in\cC_i(\eta)$, then
\begin{equation}\label{eq:coarse}
 |A|^2\le\frac2nH^2,
\end{equation}
and
\begin{equation}\label{eq:shape-square-bound}
 \sum_\alpha(A_\alpha^\circ)^2 \le\begin{cases} (1-\eta)H^2\Id/n^2,&i=1,\\
 2(1-\eta)H^2\Id/[n(n-2)],&i=2.
 \end{cases}
\end{equation}
\end{lemma}

\begin{proof}
For every unit vector $u$, the identity $\tr A_\alpha^\circ=0$ and Cauchy--Schwarz give $|A_\alpha^\circ u|^2\le\frac{n-1}{n}|A_\alpha^\circ|^2$. By \eqref{eq:phi1},
\[
 |A^\circ|^2\le-\Phi_u^{1,\eta}(A^\circ) =|A^\circ|^2+n\sum_\alpha|A_\alpha^\circ u|^2 \le|A^\circ|^2+(n-1)|A^\circ|^2=n|A^\circ|^2.
\]
For the second pair of inequalities, fix $\alpha$ and choose an orthonormal eigenbasis $e_1,\ldots,e_n$ of $A_\alpha^\circ$, with eigenvalues $\lambda_1,\ldots,\lambda_n$. Write $u=\sum_{j<k}u_{jk}e_j\wedge e_k$, where $\sum_{j<k}u_{jk}^2=1$. Since $\sum_j\lambda_j=0$, we have $(\lambda_j+\lambda_k)^2\le(n-2)\sum_{l\ne j,k}\lambda_l^2$ by Cauchy--Schwarz. Thus
\begin{align*}
 (n-4)\ip{((A_\alpha^\circ)^2)^{[2]}u}u+2|(A_\alpha^\circ)^{[2]}u|^2 &=\sum_{j<k}\bigl[(n-4)(\lambda_j^2+\lambda_k^2)+2(\lambda_j+\lambda_k)^2\bigr]u_{jk}^2\\
 &\le \sum_{j<k}\bigl[(n-2)(\lambda_j^2+\lambda_k^2)+(\lambda_j+\lambda_k)^2\bigr]u_{jk}^2\\
 &\le(n-2)|A_\alpha^\circ|^2.
\end{align*}
Summing over $\alpha$ and substituting in \eqref{eq:phi2}, we obtain
\[
 \frac{2(2n-1)}{n-1}|A^\circ|^2 \le-\Phi_u^{2,\eta}(A^\circ) \le\left(\frac{2(2n-1)}{n-1}+n-2\right)|A^\circ|^2 =\frac{n(n+1)}{n-1}|A^\circ|^2.
\]
Let $A\in\cC_1(\eta)$, and let $v\in V$ be a unit vector. Since $\Phi_v^{1,\eta}(A)\ge0$ and $\frac n{n-1}\sum_\alpha|A_\alpha^\circ v|^2\le|A^\circ|^2$, we obtain
\[
 \frac{n^2}{n-1}\sum_\alpha|A_\alpha^\circ v|^2 \le |A^\circ|^2+n\sum_\alpha|A_\alpha^\circ v|^2 \le\frac{1-\eta}{n-1}H^2.
\]
This proves \eqref{eq:shape-square-bound} for $i=1$. If $A\in\cC_2(\eta)$, let $v\in V$ be a unit vector and choose an orthonormal basis $e_1,\ldots,e_n$ with $e_1=v$. The identity $\tr A_\alpha^\circ=0$ gives
\begin{align*}
 \sum_{j>1}\ip{((A_\alpha^\circ)^2)^{[2]}(e_1\wedge e_j)} {e_1\wedge e_j} &=|A_\alpha^\circ|^2+(n-2)|A_\alpha^\circ e_1|^2,\\
 \sum_{j>1}|(A_\alpha^\circ)^{[2]}(e_1\wedge e_j)|^2 &=|A_\alpha^\circ|^2+(n-4)|A_\alpha^\circ e_1|^2.
\end{align*}
Summing $\Phi_{e_1\wedge e_j}^{2,\eta}(A)\ge0$ over $j>1$ gives
\begin{align*}
 \frac{n^3}{n-1}\sum_\alpha|A_\alpha^\circ v|^2 &\le(5n-4)|A^\circ|^2+n(n-4)\sum_\alpha|A_\alpha^\circ v|^2\\
 &=\frac{2(n+1)(1-\eta)}{n-2}H^2 -\sum_{j>1}\Phi_{e_1\wedge e_j}^{2,\eta}(A)\\
 &\le\frac{2(n+1)(1-\eta)}{n-2}H^2.
\end{align*}
Since $n^3/(n-1)\ge n(n+1)$, \eqref{eq:shape-square-bound} follows for $i=2$ as well.

For $A\in\cC_1(\eta)$, taking the trace of \eqref{eq:shape-square-bound} gives $|A^\circ|^2\le(1-\eta)H^2/n$. For $A\in\cC_2(\eta)$, \eqref{eq:phi2} gives
\[
 |A^\circ|^2\le\frac{(n+1)(1-\eta)}{(n-2)(2n-1)}H^2 \le\frac{1-\eta}{n}H^2.
\]
Hence $|A|^2=|A^\circ|^2+H^2/n\le2H^2/n$ in both cases, proving \eqref{eq:coarse}.
\end{proof}

For $Z\in\Sym^3(V)\otimes N$, write
\[
Z_j=Z(e_j,\cdot,\cdot), \quad |\tr Z|^2=\sum_j|\tr Z_j|^2.
\]

\begin{lemma}\label{lem:codazzi}
There is a constant $\chi_n>0$ such that, for every $A\in\Sym^2(V)\otimes N$, every $Z\in\Sym^3(V)\otimes N$, and every unit $u$ in the respective space,
\[
 -\sum_jD^2\Phi_u^{i,\eta}(A)[Z_j,Z_j]\ge\chi_n|Z|^2, \qquad i=1,2.
\]
\end{lemma}

\begin{proof}
The proof of \cite[Proposition~6, (30a)]{AB} gives
\begin{equation}\label{eq:codazzi-trace}
 |\tr Z|^2\le\frac{n+2}{3}|Z|^2,\qquad \sum_j|Z_j^\circ|^2=|Z|^2-\frac1n|\tr Z|^2.
\end{equation}
By \eqref{eq:phi1-hessian} and \eqref{eq:codazzi-trace},
\begin{align*}
 -\frac12\sum_jD^2\Phi_u^{1,\eta}(A)[Z_j,Z_j] &\ge\sum_j|Z_j^\circ|^2-\frac{1-\eta}{n-1}|\tr Z|^2\\
 &\ge|Z|^2-\left(\frac1n+\frac1{n-1}\right)|\tr Z|^2 \ge\frac{n^2-6n+2}{3n(n-1)}|Z|^2.
\end{align*}
For $i=2$, \eqref{eq:phi2-hessian} and $2(2n-1)/(n-1)\ge2(n+1)/(n-2)$ give
\begin{align*}
 -\frac12\sum_jD^2\Phi_u^{2,\eta}(A)[Z_j,Z_j] &\ge\frac{2(n+1)}{n-2} \left(\sum_j|Z_j^\circ|^2-\frac{1-\eta}{n-1}|\tr Z|^2\right)\\
 &\ge\frac{2(n+1)}{n-2}\frac{n^2-6n+2}{3n(n-1)}|Z|^2.
\end{align*}
The assertion follows since $n\ge10$.
\end{proof}

We next estimate $D\Phi_u^{i,\eta}(A)[\Q(A)]$. We begin with the case $H=0$.

\begin{lemma}\label{lem:reaction-zero-mean}
If $H=0$, then for every unit $u$,
\[
 D\Phi_u^{i,\eta}(A)[\Q(A)] \ge2(n-1)|A|^2\Phi_u^{i,\eta}(A), \qquad i=1,2.
\]
If $\Phi_u^{i,\eta}(A)=D\Phi_u^{i,\eta}(A)[\Q(A)]=0$, then $A=0$.
\end{lemma}

\begin{proof}
Fix $\beta$ and choose an orthonormal eigenbasis of $A_\beta$, with eigenvalues $\lambda_1,\ldots,\lambda_n$. Cauchy--Schwarz and $(\lambda_j-\lambda_k)^2\le2|A_\beta|^2$ give
\begin{align*}
 \left|\tr(A_\alpha A_\beta)A_\beta +A_\beta^2A_\alpha+A_\alpha A_\beta^2 -2A_\beta A_\alpha A_\beta\right|^2 &=|A_\beta|^2\left(\sum_j\lambda_j(A_\alpha)_{jj}\right)^2 +\sum_{j\ne k}(\lambda_j-\lambda_k)^4(A_\alpha)_{jk}^2\\
 &\le |A_\beta|^4\left(\sum_j(A_\alpha)_{jj}^2 +4\sum_{j\ne k}(A_\alpha)_{jk}^2\right)\\
 &\le4|A_\beta|^4|A_\alpha|^2.
\end{align*}
Using \eqref{eq:reaction} and the triangle inequality, we obtain
\[
 |\Q(A)|^2=\sum_\alpha|\Q(A)_\alpha|^2 \le\sum_\alpha\left(2\sum_\beta|A_\beta|^2|A_\alpha|\right)^2 =4|A|^6.
\]
Since $\tr A=\tr\Q(A)=0$, Lemma~\ref{lem:scalar} gives
\begin{align*}
 -\Phi_u^{1,\eta}(\Q(A)) &\le4n|A|^6 \le4n|A|^4[-\Phi_u^{1,\eta}(A)],\\
 -\Phi_u^{2,\eta}(\Q(A)) &\le\frac{4n(n+1)}{n-1}|A|^6 \le4n|A|^4[-\Phi_u^{2,\eta}(A)].
\end{align*}
Since $\Phi_u^{i,\eta}(A)\le0$ and $4\sqrt n\le2(n-1)$, Cauchy--Schwarz in the inner product of Lemma~\ref{lem:phi-derivatives}(iii) gives
\begin{align*}
 D\Phi_u^{i,\eta}(A)[\Q(A)] &=D^2\Phi_u^{i,\eta}(A)[A,\Q(A)]\\
 &\ge-\sqrt{[-D^2\Phi_u^{i,\eta}(A)[A,A]] [-D^2\Phi_u^{i,\eta}(A)[\Q(A),\Q(A)]]}\\
 &=-2\sqrt{[-\Phi_u^{i,\eta}(A)][-\Phi_u^{i,\eta}(\Q(A))]}\\
 &\ge4\sqrt n\,|A|^2\Phi_u^{i,\eta}(A) \ge2(n-1)|A|^2\Phi_u^{i,\eta}(A).
\end{align*}
If $\Phi_u^{i,\eta}(A)=0$, Lemma~\ref{lem:scalar} gives $|A|^2\le-\Phi_u^{i,\eta}(A)=0$.
\end{proof}

\subsection{The cone \texorpdfstring{$\cC_1$}{C1}}
Assume $H>0$. We use the decomposition
\[
 A=\frac{A^{\bH}}{H^2}\otimes\bH+\widehat A,
 \qquad
 |A|^2=\frac{|A^{\bH}|^2}{H^2}+|\widehat A|^2,
 \qquad
 A^{\bH} = \langle A, \bH\rangle_N.
\]
The tensor $\widehat A$ is traceless and takes values in $\bH^\perp$.
Fix an orthonormal normal frame $\nu_1=\bH/H,\nu_2,\ldots,\nu_q$.
Then $A_1=A^{\bH}/H$, $A_1^\circ=(A^{\bH})^\circ/H$, and
$A_\alpha=A_\alpha^\circ=\widehat A_\alpha$ for $\alpha>1$. Thus
\[
 A^\circ=\frac{(A^{\bH})^\circ}{H^2}\otimes\bH+\widehat A,
 \qquad |A^\circ|^2=\frac{|(A^{\bH})^\circ|^2}{H^2}+|\widehat A|^2.
\]
Substituting $A_1=A_1^\circ+(H/n)\Id$ in \eqref{eq:reaction} gives
\begin{equation}\label{eq:mean-orthogonal-Q}
\begin{aligned}
 \tr\Q(A)=\frac{|A^{\bH}|^2}{H^2}\bH
       +\sum_{\alpha>1}\tr(\widehat A_\alpha A^{\bH})\nu_\alpha,
       \qquad
 \Q(A)^\circ=\Q(A^\circ)
             +\frac1n(A^{\bH})^\circ\otimes\bH.
\end{aligned}
\end{equation}

\begin{lemma}\label{lem:first-algebra}
Let $A\in\Sym^2(V)\otimes N$ with $H>0$, and let $u\in V$ be a unit vector such that $\Phi_u^{1,\eta}(A) = \min_{|v|=1}\Phi_v^{1,\eta}(A)$. Then
\[
 \frac12D\Phi_u^{1,0}(A^\circ)[\Q(A^\circ)] \ge\Phi_u^{1,0}(A^\circ) \left(\frac{|(A^{\bH})^\circ|^2}{H^2} -\frac{n-1}{n}\Phi_u^{1,0}(\widehat A)\right).
\]
Equality holds only if $\widehat A=0$.
\end{lemma}

\begin{proof}
It suffices to consider $A^\circ\ne0$. By homogeneity, we may assume $\sum_\alpha|A_\alpha^\circ u|^2=1$. The choice of $u$ then gives
\[
 \sum_\alpha(A_\alpha^\circ)^2\le\Id,\qquad \sum_\alpha(A_\alpha^\circ)^2u=u,\qquad |A^\circ|^2\le n.
\]
In an orthonormal eigenbasis of $A_1^\circ$, Cauchy--Schwarz gives, for every $\alpha>1$,
\begin{align*}
 |A_1^\circ A_\alpha^\circ-A_\alpha^\circ A_1^\circ|^2 &\le2|A_1^\circ|^2|A_\alpha^\circ|^2 -2\tr(A_1^\circ A_\alpha^\circ)^2,\\
 \tr(A_1^\circ A_\alpha^\circ)^2 &\le\frac n4|A_1^\circ A_\alpha^\circ+A_\alpha^\circ A_1^\circ|^2.
\end{align*}
Expanding the squared norms, applying Cauchy--Schwarz for $\alpha,\beta>1$, and retaining $\beta=1$ and $\beta=\alpha$ in the subtracted sum, we obtain
\begin{align*}
 &\sum_{\alpha,\beta}\tr(A_\alpha^\circ A_\beta^\circ)^2 +\sum_{\alpha,\beta} |A_\alpha^\circ A_\beta^\circ-A_\beta^\circ A_\alpha^\circ|^2\\
 &\le |A^\circ|^4 +4\sum_{\alpha>1}\tr\left((A_\alpha^\circ)^2 \sum_\beta(A_\beta^\circ)^2\right) -\sum_{\alpha>1,\beta} |A_\alpha^\circ A_\beta^\circ+A_\beta^\circ A_\alpha^\circ|^2\\
 &\le |A^\circ|^4+4|\widehat A|^2 -\frac4n\sum_{\alpha>1}\tr(A_1^\circ A_\alpha^\circ)^2 -4\sum_{\alpha>1}|(A_\alpha^\circ)^2|^2.
\end{align*}
For $\alpha>1$, Young's inequality gives
\begin{align*}
 &2n\tr(A_1^\circ A_\alpha^\circ) \ip{A_1^\circ u}{A_\alpha^\circ u} -\frac4n\tr(A_1^\circ A_\alpha^\circ)^2\\
 &\le2(n-2)|A_1^\circ|\,|A_\alpha^\circ| |A_1^\circ u|\,|A_\alpha^\circ u| +n|A_1^\circ u|^2|A_\alpha^\circ u|^2\\
 &\le2(n-2)|A_1^\circ|^2|A_\alpha^\circ u|^2 +\frac{n-2}{2}|A_\alpha^\circ|^2|A_1^\circ u|^2 +n|A_1^\circ u|^2|A_\alpha^\circ u|^2.
\end{align*}
Since $\sum_\beta(A_\beta^\circ)^2\le\Id$ and $\sum_\alpha|A_\alpha^\circ u|^2=1$, we have $\sum_{\alpha,\beta}|A_\beta^\circ A_\alpha^\circ u|^2\le1$. Hence
\begin{align*}
 \sum_{\alpha,\beta}\left\langle A_\alpha^\circ u, \bigl((A_\beta^\circ)^2A_\alpha^\circ +A_\alpha^\circ(A_\beta^\circ)^2 -2A_\beta^\circ A_\alpha^\circ A_\beta^\circ\bigr)u \right\rangle &=1+3\sum_{\alpha,\beta}|A_\beta^\circ A_\alpha^\circ u|^2 -\sum_{\alpha,\beta} |(A_\alpha^\circ A_\beta^\circ+A_\beta^\circ A_\alpha^\circ)u|^2\\
 &\le4-4|(A_1^\circ)^2u|^2 \le4-4|A_1^\circ u|^4 \le8|\widehat A(u,\cdot)|^2.
\end{align*}
Substituting these inequalities in \eqref{eq:phi1-derivative} and using \eqref{eq:reaction}, Cauchy--Schwarz for $\alpha,\beta>1$, and $|A_1^\circ u|^2=1-|\widehat A(u,\cdot)|^2$, we obtain
\begin{align*}
 \frac12D\Phi_u^{1,0}(A^\circ)[\Q(A^\circ)] &=-\sum_\alpha\tr(A_\alpha^\circ\Q(A^\circ)_\alpha) -n\sum_\alpha\ip{A_\alpha^\circ u}{\Q(A^\circ)_\alpha u}\\
 &\ge-|A^\circ|^4-n|A_1^\circ|^2-\frac{n+6}{2}|\widehat A|^2 +4\sum_{\alpha>1}|(A_\alpha^\circ)^2|^2\\
 &\quad-\left((n-4)|A_1^\circ|^2+\frac{n+2}{2}|\widehat A|^2 +9n\right)|\widehat A(u,\cdot)|^2\\
 &\ge\Phi_u^{1,0}(A^\circ) \left(\frac{|(A^{\bH})^\circ|^2}{H^2} -\frac{n-1}{n}\Phi_u^{1,0}(\widehat A)\right)\\
 &\quad+\left(1-\frac{|A^\circ|^2}{n}\right)|\widehat A|^2 +\bigl(3|A^\circ|^2+n(n-10)\bigr)|\widehat A(u,\cdot)|^2 +4\sum_{\alpha>1}|(A_\alpha^\circ)^2|^2.
\end{align*}
Since $n\ge10$ and $|A^\circ|^2\le n$, the last three terms are nonnegative. If equality holds, then $\sum_{\alpha>1}|(A_\alpha^\circ)^2|^2=0$, and hence $\widehat A=0$.
\end{proof}

\begin{proposition}\label{prop:reaction-first}
Let $A\in\Sym^2(V)\otimes N$ with $H>0$, and let $u\in V$ be a unit vector such that $\Phi_u^{1,\eta}(A)=\min_{|v|=1}\Phi_v^{1,\eta}(A)$. Then
\[
 D\Phi_u^{1,\eta}(A)[\Q(A)] \ge2(n-1)|A|^2\min\{\Phi_u^{1,\eta}(A),0\}.
\]
If $\Phi_u^{1,\eta}(A)=D\Phi_u^{1,\eta}(A)[\Q(A)]=0$, then $\widehat A=0$.
\end{proposition}

\begin{proof}
Equations~\eqref{eq:phi1-derivative} and \eqref{eq:mean-orthogonal-Q} give
\begin{align*}
 \frac12D\Phi_u^{1,\eta}(A)[\Q(A)] &=\frac{|A^{\bH}|^2}{H^2}\Phi_u^{1,\eta}(A) +\frac{H^2}{n}[-\Phi_u^{1,0}(\widehat A)]\\
 &\quad+\frac{|(A^{\bH})^\circ|^2}{H^2}[-\Phi_u^{1,0}(A^\circ)] +\frac12D\Phi_u^{1,0}(A^\circ)[\Q(A^\circ)].
\end{align*}
Lemma~\ref{lem:first-algebra} and \eqref{eq:phi1} then give
\begin{equation}\label{eq:first-sharp-reaction}
\begin{aligned}
 \frac12D\Phi_u^{1,\eta}(A)[\Q(A)] &\ge\frac{|A^{\bH}|^2}{H^2}\Phi_u^{1,\eta}(A) +\left(\frac{H^2}{n}+\frac{n-1}{n}\Phi_u^{1,0}(A^\circ)\right) [-\Phi_u^{1,0}(\widehat A)]\\
 &=\frac{|A^{\bH}|^2}{H^2}\Phi_u^{1,\eta}(A) +\frac{n-1}{n}\Phi_u^{1,0}(A)[-\Phi_u^{1,0}(\widehat A)].
\end{aligned}
\end{equation}
Since $\Phi_u^{1,0}(A)=\Phi_u^{1,\eta}(A)+\eta H^2/(n-1)$ and Lemma~\ref{lem:scalar} gives
\[
 0\le\frac{|A^{\bH}|^2}{H^2} +\frac{n-1}{n}[-\Phi_u^{1,0}(\widehat A)] \le\frac{|A^{\bH}|^2}{H^2}+(n-1)|\widehat A|^2 \le(n-1)|A|^2,
\]
we conclude that
\begin{align*}
 D\Phi_u^{1,\eta}(A)[\Q(A)] &\ge2\left(\frac{|A^{\bH}|^2}{H^2} +\frac{n-1}{n}[-\Phi_u^{1,0}(\widehat A)]\right)\Phi_u^{1,\eta}(A)\\
 &\ge2(n-1)|A|^2\min\{\Phi_u^{1,\eta}(A),0\}.
\end{align*}
If $\Phi_u^{1,\eta}(A)=D\Phi_u^{1,\eta}(A)[\Q(A)]=0$, equality holds in Lemma~\ref{lem:first-algebra}, so $\widehat A=0$.
\end{proof}

\subsection{The cone \texorpdfstring{$\cC_2$}{C2}}

Assume $H>0$. We again write $A=(A^{\bH}/H^2)\otimes\bH+\widehat A$. Choose an orthonormal normal frame $\nu_1=\bH/H,\nu_2,\ldots,\nu_q$ such that $\tr(A_\alpha A_\beta)=0$ for distinct $\alpha,\beta>1$. This is possible by diagonalizing the symmetric matrix $(\tr(A_\alpha A_\beta))_{\alpha,\beta>1}$ through an orthogonal change of basis in $\bH^\perp$.

\begin{lemma}\label{lem:second-mixed}
Let $A\in\Sym^2(V)\otimes N$ with $H>0$, and let $u\in\Lambda^2V$ be a unit two-form such that $\Phi_u^{2,\eta}(A) = \min_{|v|=1}\Phi_v^{2,\eta}(A)$. Then
\begin{align*}
 &\sum_{\alpha>1}\left(\tr((A^{\bH})^\circ \widehat A_\alpha)^2 +|(A^{\bH})^\circ\widehat A_\alpha -\widehat A_\alpha(A^{\bH})^\circ|^2\right)\\
 &\le |(A^{\bH})^\circ|^2 \left(\frac{10}{9}|\widehat A|^2 +\frac3{16}\left((n-4)\sum_j |\widehat A\iota_{e_j}u|^2 +2|\widehat A^{[2]}u|^2\right)\right).
\end{align*}
\end{lemma}

\begin{proof}
Choose an orthonormal eigenbasis $f_1,\ldots,f_n$ of $(A^{\bH})^\circ$, with eigenvalues $\lambda_1,\ldots,\lambda_n$. We choose $v=f_1\wedge f_2$ to minimize $\Phi_w^{2,0}((A^{\bH}/H^2)\otimes\bH)$ among unit two-forms $w$. To do this, relabel the basis so that
\begin{equation}\label{eq:second-maxpair}
 (n-2)(\lambda_1^2+\lambda_2^2)+4\lambda_1\lambda_2 \ge(n-2)(\lambda_k^2+\lambda_l^2)+4\lambda_k\lambda_l \qquad(k<l).
\end{equation}
For any unit two-form $w=\sum_{k<l}w_{kl}f_k\wedge f_l$, we have $\sum_{k<l}w_{kl}^2=1$, and \eqref{eq:phi2} gives
\begin{align*}
 &\Phi_w^{2,0}((A^{\bH}/H^2)\otimes\bH) -\Phi_v^{2,0}((A^{\bH}/H^2)\otimes\bH)\\
 &=\frac1{H^2}\sum_{k<l}\bigl[(n-2)(\lambda_1^2+\lambda_2^2 -\lambda_k^2-\lambda_l^2) +4(\lambda_1\lambda_2-\lambda_k\lambda_l)\bigr]w_{kl}^2\\
 &\ge0.
\end{align*}
Hence $v$ minimizes $\Phi_w^{2,0}((A^{\bH}/H^2)\otimes\bH)$ among unit two-forms $w$. For distinct $i,j>2$, adding \eqref{eq:second-maxpair} for $(k,l)=(1,i),(1,j),(2,i),(2,j)$ and dividing by two gives
\[
 (n-2)(\lambda_1^2+\lambda_2^2-\lambda_i^2-\lambda_j^2) +8\lambda_1\lambda_2-2(\lambda_1+\lambda_2)(\lambda_i+\lambda_j)\ge0.
\]
Since $\sum_k\lambda_k=0$, Cauchy--Schwarz gives $\sum_{k\notin\{1,2,i,j\}}\lambda_k^2 \ge(\lambda_1+\lambda_2+\lambda_i+\lambda_j)^2/(n-4)$. Hence
\begin{align*}
 &(\lambda_i-\lambda_j)^2\\
 &\le\frac n{n-1}|(A^{\bH})^\circ|^2 -\frac3{n-1}(\lambda_1-\lambda_2)^2-\frac{[2(\lambda_1+\lambda_2)+(n-2)(\lambda_i+\lambda_j)]^2}{(n-1)(n-4)}\\
 &\le\frac{10}{9}|(A^{\bH})^\circ|^2.
\end{align*}

For $i>2$, taking $(k,l)=(1,i)$ in \eqref{eq:second-maxpair} gives
\[
 (\lambda_2-\lambda_i)[4\lambda_1+(n-2)(\lambda_2+\lambda_i)]\ge0.
\]
If $(\lambda_1-\lambda_i)(\lambda_2-\lambda_i)\le0$, then
\[
 |\lambda_1-\lambda_i|^2 \le\left|\lambda_1-\frac{\lambda_2+\lambda_i}{2}\right|^2 \le\frac32|(A^{\bH})^\circ|^2.
\]
Otherwise, $(\lambda_1-\lambda_i)[4\lambda_1+(n-2)(\lambda_2+\lambda_i)]\ge0$. Using $\sum_k\lambda_k=0$ and Cauchy--Schwarz, we obtain
\begin{align*}
 |\lambda_1-\lambda_i|^2 &\le\left|\lambda_1-\lambda_i+ \frac{4\lambda_1+(n-2)(\lambda_2+\lambda_i)}{2n}\right|^2\\
 &=\left|\frac{n+1}{n}\lambda_1 +\frac{n-4}{2n}\lambda_2-\frac{n+4}{2n}\lambda_i -\frac1n\sum_{k\notin\{1,2,i\}}\lambda_k\right|^2\\
 &\le\left(\frac32+\frac3n+\frac6{n^2}\right)|(A^{\bH})^\circ|^2\\
&\le\left(\frac{10}{9}+\frac{3(n-2)}{32}\right)|(A^{\bH})^\circ|^2.
\end{align*}
Since $n\ge10$, this also bounds $|\lambda_1-\lambda_i|^2$ in the first case. Taking $(k,l)=(2,i)$ in \eqref{eq:second-maxpair} gives the same bound with $\lambda_1$ replaced by $\lambda_2$. We also have $(\lambda_1-\lambda_2)^2\le2|(A^{\bH})^\circ|^2 \le(10/9+3(n-4)/16)|(A^{\bH})^\circ|^2$.

The difference $\Phi_w^{2,\eta}(A)-\Phi_w^{2,0}(A)$ is independent of $w$, so $\Phi_u^{2,0}(A)\le\Phi_v^{2,0}(A)$ by the hypothesis. By the choice of $v$,
\begin{align*}
 &\Phi_v^{2,0}(\widehat A)-\Phi_u^{2,0}(\widehat A)\\
 &=\Phi_v^{2,0}(A)-\Phi_u^{2,0}(A)+\Phi_u^{2,0}((A^{\bH}/H^2)\otimes\bH)-\Phi_v^{2,0}((A^{\bH}/H^2)\otimes\bH)\\
 &\ge0.
\end{align*}
Write $(\widehat A_\alpha)_{ij}=\ip{\widehat A_\alpha f_j}{f_i}$. Applying the bounds for $(\lambda_i-\lambda_j)^2$, and then using $\Phi_v^{2,0}(\widehat A)\ge\Phi_u^{2,0}(\widehat A)$ in \eqref{eq:phi2}, we obtain
\begin{align*}
 &\sum_{\alpha>1}\left(\tr((A^{\bH})^\circ\widehat A_\alpha)^2 +|(A^{\bH})^\circ\widehat A_\alpha -\widehat A_\alpha(A^{\bH})^\circ|^2\right)\\
 &=\sum_{\alpha>1}\left(\sum_i\lambda_i(\widehat A_\alpha)_{ii}\right)^2 +2\sum_{\alpha>1,i<j}(\lambda_i-\lambda_j)^2(\widehat A_\alpha)_{ij}^2\\
 &\le|(A^{\bH})^\circ|^2\left[\frac{10}{9}|\widehat A|^2 +\frac3{16}\sum_{\alpha>1}\left(2(n-4)(\widehat A_\alpha)_{12}^2 +(n-2)\sum_{i>2}((\widehat A_\alpha)_{1i}^2 +(\widehat A_\alpha)_{2i}^2)\right)\right]\\
 &\le|(A^{\bH})^\circ|^2\left[\frac{10}{9}|\widehat A|^2 +\frac3{16}\left((n-4)\sum_k|\widehat A\iota_{f_k}u|^2 +2|\widehat A^{[2]}u|^2\right)\right].
\end{align*}
\end{proof}

\begin{samepage}
\begin{lemma}\label{lem:second-operator-sum}
Let $A\in\Sym^2(V)\otimes N$ with $H>0$, and let $u\in\Lambda^2V$ be a unit two-form such that
\begin{equation}\label{eq:second-normalization}
\begin{gathered}
 (n-4)\sum_\alpha((A_\alpha^\circ)^2)^{[2]} +2\sum_\alpha((A_\alpha^\circ)^{[2]})^2\le\Id,\\
 (n-4)\sum_{\alpha,j}|A_\alpha^\circ\iota_{e_j}u|^2 +2\sum_\alpha|(A_\alpha^\circ)^{[2]}u|^2=1.
\end{gathered}
\end{equation}
Then
\begin{align*}
 &(n-4)\sum_{\alpha,j} \ip{A_\alpha^\circ\iota_{e_j}u}{\Q(A^\circ)_\alpha\iota_{e_j}u} +2\sum_\alpha \ip{(A_\alpha^\circ)^{[2]}u}{(\Q(A^\circ)_\alpha)^{[2]}u}\\
 &\le |A_1^\circ|^2+ \left(\frac{11}{8}+\frac58|A_1^\circ|^2-\frac12|\widehat A|^2\right)|\widehat A|^2\\
 &\quad+\left(\frac{13}{40}-\frac15|A_1^\circ|^2 +\frac65|\widehat A|^2\right) \left((n-4)\sum_j|\widehat A\iota_{e_j}u|^2 +2|\widehat A^{[2]}u|^2\right).
\end{align*}
\end{lemma}
\end{samepage}

\begin{proof}
Choose an orthonormal eigenbasis $e_1,\ldots,e_n$ of $A_1^\circ$, with eigenvalues $\lambda_1\le\cdots\le\lambda_n$. For every $\alpha>1$, we have
\begin{equation}\label{eq:second-entrywise-estimates}
\begin{aligned}
 \tr(A_1^\circ A_\alpha^\circ)^2 &=\left(\sum_k\lambda_k(A_\alpha^\circ)_{kk}\right)^2 \le |A_1^\circ|^2\sum_k(A_\alpha^\circ)_{kk}^2,\\
 \|A_\alpha^\circ A_1^\circ-A_1^\circ A_\alpha^\circ\|_{\op}^2 &\le\frac12\sum_{k\ne l}(\lambda_k-\lambda_l)^2(A_\alpha^\circ)_{kl}^2 \le\frac{(\lambda_n-\lambda_1)^2}{2} \sum_{k\ne l}(A_\alpha^\circ)_{kl}^2,\\
 \|(A_\alpha^\circ A_1^\circ-A_1^\circ A_\alpha^\circ)^{[2]}\|_{\op}^2 &\le\left(\sum_k|\lambda_k| \left\|(A_\alpha^\circ e_k\otimes e_k-e_k\otimes A_\alpha^\circ e_k)^{[2]} \right\|_{\op}\right)^2\\
 &=\left(\sum_k|\lambda_k| \left(\sum_{l\ne k}(A_\alpha^\circ)_{kl}^2\right)^{1/2}\right)^2 \le |A_1^\circ|^2\sum_{k\ne l}(A_\alpha^\circ)_{kl}^2.
\end{aligned}
\end{equation}
Evaluating \eqref{eq:second-normalization} on $e_1\wedge e_n$ gives
\[
 3(\lambda_n-\lambda_1)^2 \le\frac{n-4}{2}(\lambda_n-\lambda_1)^2 \le(n-4)(\lambda_1^2+\lambda_n^2)+2(\lambda_1+\lambda_n)^2\le1.
\]
Moreover,
\[
 2|A_1^\circ|^2-(\lambda_n-\lambda_1)^2 =(\lambda_1+\lambda_n)^2+2\sum_{1<k<n}\lambda_k^2\ge0.
\]
Since $|(A_1^\circ)^{[2]}u|^2\le2\sum_j|A_1^\circ\iota_{e_j}u|^2$ and $n\ge10$,
\[
 5|(A_1^\circ)^{[2]}u|^2 \le(n-4)\sum_j|A_1^\circ\iota_{e_j}u|^2 +2|(A_1^\circ)^{[2]}u|^2.
\]
Since the commutators are skew-symmetric, the cross terms vanish when we expand the following squared norms. Using \eqref{eq:second-entrywise-estimates} and the three inequalities above, we obtain
\begin{align}
 &(n-4)\sum_j\left| (\tr(A_1^\circ A_\alpha^\circ)\Id +2A_\alpha^\circ A_1^\circ-2A_1^\circ A_\alpha^\circ) A_1^\circ\iota_{e_j}u\right|^2\notag\\*
 &\quad+2\left| (\tr(A_1^\circ A_\alpha^\circ)\Id +2(A_\alpha^\circ A_1^\circ-A_1^\circ A_\alpha^\circ)^{[2]}) (A_1^\circ)^{[2]}u\right|^2\notag\\
 &\le |A_1^\circ|^2\sum_k(A_\alpha^\circ)_{kk}^2 \left((n-4)\sum_j|A_1^\circ\iota_{e_j}u|^2 +2|(A_1^\circ)^{[2]}u|^2\right)\notag\\*
 &\quad+\sum_{k\ne l}(A_\alpha^\circ)_{kl}^2 \left(2(n-4)(\lambda_n-\lambda_1)^2 \sum_j|A_1^\circ\iota_{e_j}u|^2 +8|A_1^\circ|^2|(A_1^\circ)^{[2]}u|^2\right)\notag\\
 &=\left[|A_1^\circ|^2\sum_k(A_\alpha^\circ)_{kk}^2 +2(\lambda_n-\lambda_1)^2\sum_{k\ne l}(A_\alpha^\circ)_{kl}^2\right]\cdot\left((n-4)\sum_j|A_1^\circ\iota_{e_j}u|^2 +2|(A_1^\circ)^{[2]}u|^2\right)\notag\\*
 &\quad+4\left(2|A_1^\circ|^2-(\lambda_n-\lambda_1)^2\right) \sum_{k\ne l}(A_\alpha^\circ)_{kl}^2|(A_1^\circ)^{[2]}u|^2\notag\\
 &\le\left[|A_1^\circ|^2\sum_k(A_\alpha^\circ)_{kk}^2 +\frac{6(\lambda_n-\lambda_1)^2+8|A_1^\circ|^2}{5} \sum_{k\ne l}(A_\alpha^\circ)_{kl}^2\right]\notag\\*
 &\qquad\cdot\left((n-4)\sum_j|A_1^\circ\iota_{e_j}u|^2 +2|(A_1^\circ)^{[2]}u|^2\right)\notag\\
 &\le\frac25(1+4|A_1^\circ|^2)|A_\alpha^\circ|^2 \left((n-4)\sum_j|A_1^\circ\iota_{e_j}u|^2 +2|(A_1^\circ)^{[2]}u|^2\right).
\label{eq:second-weighted-estimate}
\end{align}
\enlargethispage{\baselineskip}

For $\alpha,\beta>1$, the calculation in \eqref{eq:second-entrywise-estimates}, performed in an orthonormal eigenbasis of $A_\beta^\circ$, gives
\[
 \max\left\{ \|A_\alpha^\circ A_\beta^\circ-A_\beta^\circ A_\alpha^\circ\|_{\op}, \|(A_\alpha^\circ A_\beta^\circ-A_\beta^\circ A_\alpha^\circ)^{[2]}\|_{\op} \right\}\le |A_\alpha^\circ|\,|A_\beta^\circ|.
\]
This calculation uses the symmetry of the matrices; the normal frame remains fixed. For distinct $\alpha,\beta>1$, we have $\tr(A_\alpha^\circ A_\beta^\circ)=0$, while for $\alpha=\beta>1$ the trace is $|A_\alpha^\circ|^2$.

By \eqref{eq:second-normalization},
\[
 \left[(n-4)\sum_\beta((A_\beta^\circ)^2)^{[2]} +2\sum_\beta((A_\beta^\circ)^{[2]})^2\right]u=u.
\]
Substituting \eqref{eq:reaction} and using this equality gives
\begin{align*}
 &(n-4)\sum_{\alpha,j}\ip{A_\alpha^\circ\iota_{e_j}u} {\Q(A^\circ)_\alpha\iota_{e_j}u} +2\sum_\alpha\ip{(A_\alpha^\circ)^{[2]}u} {(\Q(A^\circ)_\alpha)^{[2]}u}\\*
 &=(n-4)\sum_{\alpha,\beta,j} \left\langle A_\alpha^\circ\iota_{e_j}u, (\tr(A_\alpha^\circ A_\beta^\circ)\Id +2A_\alpha^\circ A_\beta^\circ-2A_\beta^\circ A_\alpha^\circ) A_\beta^\circ\iota_{e_j}u\right\rangle\\*
 &\quad+2\sum_{\alpha,\beta} \left\langle(A_\alpha^\circ)^{[2]}u, (\tr(A_\alpha^\circ A_\beta^\circ)\Id +2(A_\alpha^\circ A_\beta^\circ-A_\beta^\circ A_\alpha^\circ)^{[2]}) (A_\beta^\circ)^{[2]}u\right\rangle\\*
 &\quad-\sum_\alpha\left\langle(A_\alpha^\circ)^{[2]}u, \left[\Id-(n-4)\sum_\beta((A_\beta^\circ)^2)^{[2]} -2\sum_\beta((A_\beta^\circ)^{[2]})^2\right] (A_\alpha^\circ)^{[2]}u\right\rangle.
\end{align*}
The third term on the right is nonpositive by \eqref{eq:second-normalization}. Let $\varepsilon>0$. We use Young's inequality $2\ip vw\le\varepsilon|v|^2+\varepsilon^{-1}|w|^2$ for $\alpha=1<\beta$ and $\beta=1<\alpha$, and Cauchy--Schwarz for $\alpha,\beta>1$.

Using \eqref{eq:second-weighted-estimate}, we obtain
\begin{align}
 &(n-4)\sum_{\alpha,j}\ip{A_\alpha^\circ\iota_{e_j}u} {\Q(A^\circ)_\alpha\iota_{e_j}u} +2\sum_\alpha\ip{(A_\alpha^\circ)^{[2]}u} {(\Q(A^\circ)_\alpha)^{[2]}u}\notag\\*
 &\le |A_1^\circ|^2 \left((n-4)\sum_j|A_1^\circ\iota_{e_j}u|^2 +2|(A_1^\circ)^{[2]}u|^2\right)\notag\\*
 &\quad+2(n-4)\sum_{\alpha>1,j} \left\langle A_\alpha^\circ\iota_{e_j}u, (\tr(A_1^\circ A_\alpha^\circ)\Id +2A_\alpha^\circ A_1^\circ-2A_1^\circ A_\alpha^\circ) A_1^\circ\iota_{e_j}u\right\rangle\notag\\*
 &\quad+4\sum_{\alpha>1} \left\langle(A_\alpha^\circ)^{[2]}u, (\tr(A_1^\circ A_\alpha^\circ)\Id +2(A_\alpha^\circ A_1^\circ-A_1^\circ A_\alpha^\circ)^{[2]}) (A_1^\circ)^{[2]}u\right\rangle\notag\\*
 &\quad+2(n-4)\sum_{\alpha,\beta>1,j} |A_\alpha^\circ|\,|A_\beta^\circ| |A_\alpha^\circ\iota_{e_j}u|\,|A_\beta^\circ\iota_{e_j}u|\notag\\*
 &\quad+4\sum_{\alpha,\beta>1} |A_\alpha^\circ|\,|A_\beta^\circ| |(A_\alpha^\circ)^{[2]}u|\,|(A_\beta^\circ)^{[2]}u|\notag\\
 &\le |A_1^\circ|^2 \left((n-4)\sum_j|A_1^\circ\iota_{e_j}u|^2 +2|(A_1^\circ)^{[2]}u|^2\right)\notag\\*
 &\quad+(2|\widehat A|^2+\varepsilon) \left((n-4)\sum_j|\widehat A\iota_{e_j}u|^2 +2|\widehat A^{[2]}u|^2\right)\notag\\*
 &\quad+\frac{n-4}{\varepsilon}\sum_{\alpha>1,j}\left| (\tr(A_1^\circ A_\alpha^\circ)\Id +2A_\alpha^\circ A_1^\circ-2A_1^\circ A_\alpha^\circ) A_1^\circ\iota_{e_j}u\right|^2\notag\\*
 &\quad+\frac2\varepsilon\sum_{\alpha>1}\left| (\tr(A_1^\circ A_\alpha^\circ)\Id +2(A_\alpha^\circ A_1^\circ-A_1^\circ A_\alpha^\circ)^{[2]}) (A_1^\circ)^{[2]}u\right|^2\notag\\
 &\le\left[|A_1^\circ|^2+ \frac{2(1+4|A_1^\circ|^2)}{5\varepsilon}|\widehat A|^2\right] \left((n-4)\sum_j|A_1^\circ\iota_{e_j}u|^2 +2|(A_1^\circ)^{[2]}u|^2\right)\notag\\*
 &\quad+(2|\widehat A|^2+\varepsilon) \left((n-4)\sum_j|\widehat A\iota_{e_j}u|^2 +2|\widehat A^{[2]}u|^2\right).\label{eq:second-young-bound}
\end{align}

Taking the trace in \eqref{eq:second-normalization} gives
\begin{equation}\label{eq:second-trace-bound}
 |A_1^\circ|^2+|\widehat A|^2\le\frac{n-1}{2(n-3)}\le\frac9{14}.
\end{equation}

Set $\varepsilon=13/40+4|A_1^\circ|^2/5+|\widehat A|^2/4$. By \eqref{eq:second-trace-bound},
\begin{align*}
 &\left(\frac{11}{8}+\frac58|A_1^\circ|^2-\frac12|\widehat A|^2\right) \varepsilon-\frac25(1+4|A_1^\circ|^2)\\
 &=\frac12\left(|A_1^\circ|^2-\frac{19}{64}\right)^2 +\frac{23}{8192} +\frac{|\widehat A|^2}{160} \left[29-39\bigl(|A_1^\circ|^2+|\widehat A|^2\bigr) +19|\widehat A|^2\right]\\
 &\ge\frac12\left(|A_1^\circ|^2-\frac{19}{64}\right)^2 +\frac{23}{8192} +\frac{|\widehat A|^2}{160}\left(29-39\cdot\frac9{14}\right)>0.
\end{align*}
Since $\varepsilon>0$, this gives
\[
 \frac{2(1+4|A_1^\circ|^2)}{5\varepsilon} \le\frac{11}{8}+\frac58|A_1^\circ|^2-\frac12|\widehat A|^2.
\]
Substituting this bound in \eqref{eq:second-young-bound} and using \eqref{eq:second-normalization}, we obtain
\begin{align*}
 &(n-4)\sum_{\alpha,j}\ip{A_\alpha^\circ\iota_{e_j}u} {\Q(A^\circ)_\alpha\iota_{e_j}u} +2\sum_\alpha\ip{(A_\alpha^\circ)^{[2]}u} {(\Q(A^\circ)_\alpha)^{[2]}u}\\
 &\le\left[|A_1^\circ|^2+ \left(\frac{11}{8}+\frac58|A_1^\circ|^2-\frac12|\widehat A|^2\right) |\widehat A|^2\right] \left((n-4)\sum_j|A_1^\circ\iota_{e_j}u|^2 +2|(A_1^\circ)^{[2]}u|^2\right)\\*
 &\quad+(2|\widehat A|^2+\varepsilon) \left((n-4)\sum_j|\widehat A\iota_{e_j}u|^2 +2|\widehat A^{[2]}u|^2\right)\\
 &=\left[|A_1^\circ|^2+ \left(\frac{11}{8}+\frac58|A_1^\circ|^2-\frac12|\widehat A|^2\right) |\widehat A|^2\right] \left(1-(n-4)\sum_j|\widehat A\iota_{e_j}u|^2 -2|\widehat A^{[2]}u|^2\right)\\*
 &\quad+\left(\frac{13}{40}+\frac45|A_1^\circ|^2+\frac94|\widehat A|^2\right) \left((n-4)\sum_j|\widehat A\iota_{e_j}u|^2 +2|\widehat A^{[2]}u|^2\right)\\
 &=|A_1^\circ|^2+ \left(\frac{11}{8}+\frac58|A_1^\circ|^2-\frac12|\widehat A|^2\right) |\widehat A|^2\\*
 &\quad+\left[\frac{13}{40}-\frac15|A_1^\circ|^2 +\left(\frac78+\frac12\bigl(|A_1^\circ|^2+|\widehat A|^2\bigr) -\frac98|A_1^\circ|^2\right)|\widehat A|^2\right]\\*
 &\hspace{12mm}\cdot \left((n-4)\sum_j|\widehat A\iota_{e_j}u|^2 +2|\widehat A^{[2]}u|^2\right)\\
 &\le|A_1^\circ|^2+ \left(\frac{11}{8}+\frac58|A_1^\circ|^2-\frac12|\widehat A|^2\right) |\widehat A|^2\\*
 &\quad+\left(\frac{13}{40}-\frac15|A_1^\circ|^2+\frac65|\widehat A|^2\right) \left((n-4)\sum_j|\widehat A\iota_{e_j}u|^2 +2|\widehat A^{[2]}u|^2\right).
\end{align*}
The last inequality follows from \eqref{eq:second-trace-bound} and $7/8+9/28=67/56<6/5$.
\end{proof}

\Needspace{8\baselineskip}
\begin{lemma}\label{lem:second-directional}
Let $A\in\Sym^2(V)\otimes N$ with $H>0$, and let $u\in\Lambda^2V$ be a unit two-form such that $\Phi_u^{2,\eta}(A) = \min_{|v|=1}\Phi_v^{2,\eta}(A)$. Then
\[
 \frac12D\Phi_u^{2,0}(A^\circ)[\Q(A^\circ)] \ge\Phi_u^{2,0}(A^\circ) \left(\frac{|(A^{\bH})^\circ|^2}{H^2} -\frac{(n-1)(n-2)}{2n(n+1)}\Phi_u^{2,0}(\widehat A)\right).
\]
Equality holds only if $\widehat A=0$.
\end{lemma}

\begin{proof}
It suffices to consider $A^\circ\ne0$. By homogeneity, we may assume that \eqref{eq:second-normalization} holds. Since $\tr(\widehat A_\alpha\widehat A_\beta)=0$ for distinct $\alpha,\beta>1$, \cite[Theorem~3]{Lu} applies to the nonzero matrices $\widehat A_\alpha/|\widehat A_\alpha|$ with weights $|\widehat A_\alpha|^2$. Combining it with Lemma~\ref{lem:second-mixed}, we obtain
\begin{equation}\label{eq:second-scalar}
\begin{aligned}
 &\sum_\alpha\tr(A_\alpha^\circ \Q(A^\circ)_\alpha)\\
 &=\sum_{\alpha,\beta}\left(\tr(A_\alpha^\circ A_\beta^\circ)^2 +|A_\alpha^\circ A_\beta^\circ-A_\beta^\circ A_\alpha^\circ|^2\right)\\
 &\le |A_1^\circ|^4+\frac{20}{9}|A_1^\circ|^2|\widehat A|^2 +\frac32|\widehat A|^4 +\frac38|A_1^\circ|^2 \left((n-4)\sum_j|\widehat A\iota_{e_j}u|^2 +2|\widehat A^{[2]}u|^2\right).
\end{aligned}
\end{equation}

By \eqref{eq:phi2-derivative}, Lemma~\ref{lem:second-operator-sum}, and \eqref{eq:second-scalar},
\begin{equation}\label{eq:second-directional}
\begin{aligned}
 &-\frac12D\Phi_u^{2,0}(A^\circ)[\Q(A^\circ)]\\
 &=\frac{2(2n-1)}{n-1}\sum_\alpha \tr(A_\alpha^\circ\Q(A^\circ)_\alpha)\\
 &\quad+(n-4)\sum_{\alpha,j} \ip{A_\alpha^\circ\iota_{e_j}u}{\Q(A^\circ)_\alpha\iota_{e_j}u} +2\sum_\alpha \ip{(A_\alpha^\circ)^{[2]}u}{(\Q(A^\circ)_\alpha)^{[2]}u}\\
 &\le\frac{2(2n-1)}{n-1} \left(|A_1^\circ|^4+\frac{20}{9}|A_1^\circ|^2|\widehat A|^2 +\frac32|\widehat A|^4\right)\\
 &\quad+|A_1^\circ|^2+ \left(\frac{11}{8}+\frac58|A_1^\circ|^2-\frac12|\widehat A|^2\right) |\widehat A|^2\\
 &\quad+\left(\frac{13}{40} +\frac{26n-11}{20(n-1)}|A_1^\circ|^2 +\frac65|\widehat A|^2\right) \left((n-4)\sum_j|\widehat A\iota_{e_j}u|^2 +2|\widehat A^{[2]}u|^2\right).
\end{aligned}
\end{equation}
Since $n\ge10$,
\[
 4\le\frac{2(2n-1)}{n-1}\le\frac{38}{9},\qquad \frac{(n-1)(n-2)}{2n(n+1)}\ge\frac{18}{55},\qquad \frac{(2n-1)(n-2)}{n(n+1)}\ge\frac{76}{55}.
\]
Using these bounds and $(26n-11)/(20(n-1))\le83/60$ in \eqref{eq:second-directional}, and then \eqref{eq:second-trace-bound}, we obtain
\begin{align*}
 &\frac12D\Phi_u^{2,0}(A^\circ)[\Q(A^\circ)] -\Phi_u^{2,0}(A^\circ) \left(|A_1^\circ|^2 -\frac{(n-1)(n-2)}{2n(n+1)}\Phi_u^{2,0}(\widehat A)\right)\\
 &\ge\left(\frac3{440}+\frac{53}{3960}|A_1^\circ|^2+\frac1{990}|\widehat A|^2\right)|\widehat A|^2\\*
 &\quad+\left(\frac1{440}-\frac1{660}|A_1^\circ|^2+\frac2{11}|\widehat A|^2\right) \left((n-4)\sum_j|\widehat A\iota_{e_j}u|^2+2|\widehat A^{[2]}u|^2\right)\\
 &\ge\frac3{440}|\widehat A|^2 +\left(\frac1{770}+\frac2{11}|\widehat A|^2\right)\cdot\left((n-4)\sum_j|\widehat A\iota_{e_j}u|^2+2|\widehat A^{[2]}u|^2\right).
\end{align*}
The last two terms are nonnegative. Equality forces $\frac3{440}|\widehat A|^2=0$, and hence $\widehat A=0$.
\end{proof}

\begin{proposition}\label{prop:reaction-second}
Let $A\in\Sym^2(V)\otimes N$ with $H>0$, and let $u\in\Lambda^2V$ be a unit two-form such that $\Phi_u^{2,\eta}(A)=\min_{|v|=1}\Phi_v^{2,\eta}(A)$. Then
\[
 D\Phi_u^{2,\eta}(A)[\Q(A)] \ge2(n-1)|A|^2\min\{\Phi_u^{2,\eta}(A),0\}.
\]
If $\Phi_u^{2,\eta}(A)=D\Phi_u^{2,\eta}(A)[\Q(A)]=0$, then $\widehat A=0$.
\end{proposition}

\begin{proof}
Equations~\eqref{eq:phi2-derivative}, \eqref{eq:mean-orthogonal-Q}, and Lemma~\ref{lem:second-directional} give
\begin{equation}\label{eq:second-sharp-reaction}
\begin{aligned}
 \frac12D\Phi_u^{2,\eta}(A)[\Q(A)] &=\frac{|A^{\bH}|^2}{H^2}\Phi_u^{2,\eta}(A) +\frac{H^2}{n}[-\Phi_u^{2,0}(\widehat A)]\\
 &\quad+\frac{|(A^{\bH})^\circ|^2}{H^2}[-\Phi_u^{2,0}(A^\circ)] +\frac12D\Phi_u^{2,0}(A^\circ)[\Q(A^\circ)]\\
 &\ge\frac{|A^{\bH}|^2}{H^2}\Phi_u^{2,\eta}(A)\\
 &\quad+\left(\frac{H^2}{n} +\frac{(n-1)(n-2)}{2n(n+1)}\Phi_u^{2,0}(A^\circ)\right) [-\Phi_u^{2,0}(\widehat A)]\\
 &=\frac{|A^{\bH}|^2}{H^2}\Phi_u^{2,\eta}(A) +\frac{(n-1)(n-2)}{2n(n+1)} \Phi_u^{2,0}(A)[-\Phi_u^{2,0}(\widehat A)].
\end{aligned}
\end{equation}
Since $\Phi_u^{2,0}(A)=\Phi_u^{2,\eta}(A) +2(n+1)\eta H^2/[(n-1)(n-2)]$ and Lemma~\ref{lem:scalar} gives
\[
 0\le\frac{|A^{\bH}|^2}{H^2} +\frac{(n-1)(n-2)}{2n(n+1)}[-\Phi_u^{2,0}(\widehat A)] \le\frac{|A^{\bH}|^2}{H^2}+\frac{n-2}{2}|\widehat A|^2 \le\frac{n-2}{2}|A|^2,
\]
we conclude that
\begin{align*}
 D\Phi_u^{2,\eta}(A)[\Q(A)] &\ge2\left(\frac{|A^{\bH}|^2}{H^2} +\frac{(n-1)(n-2)}{2n(n+1)}[-\Phi_u^{2,0}(\widehat A)]\right) \Phi_u^{2,\eta}(A)\\
 &\ge(n-2)|A|^2\min\{\Phi_u^{2,\eta}(A),0\} \ge2(n-1)|A|^2\min\{\Phi_u^{2,\eta}(A),0\}.
\end{align*}
If $\Phi_u^{2,\eta}(A)=D\Phi_u^{2,\eta}(A)[\Q(A)]=0$, equality holds in Lemma~\ref{lem:second-directional}, so $\widehat A=0$.
\end{proof}

\subsection{Preservation under mean curvature flow}
The preceding estimates imply preservation under mean curvature flow.

\begin{corollary}\label{cor:preserved}
Let $F:M\times[0,T)\to\R^{n+q}$ be a smooth mean curvature flow of a closed immersion. If $A(\cdot,0)\in\cC_i(\eta)$, then $A(\cdot,t)\in\cC_i(\eta)$ for all $t\in[0,T)$.
\end{corollary}

\begin{proof}
Let $A\in\Sym^2(V)\otimes N$ and $Z\in\Sym^3(V)\otimes N$. Choose a unit $u$ in the respective space such that $\Phi_u^{i,\eta}(A)=\min_{|v|=1}\Phi_v^{i,\eta}(A)$. Lemma~\ref{lem:reaction-zero-mean}, Propositions~\ref{prop:reaction-first} and~\ref{prop:reaction-second}, and Lemma~\ref{lem:codazzi} give
\begin{equation}\label{eq:preservation-estimate}
 D\Phi_u^{i,\eta}(A)[\Q(A)] -\sum_jD^2\Phi_u^{i,\eta}(A)[Z_j,Z_j] \ge2(n-1)|A|^2\min\{\Phi_u^{i,\eta}(A),0\}.
\end{equation}
Proposition~\ref{prop:maximum} applies with $Y=\Sym^2(V)\otimes N$ and $c=2(n-1)$.
\end{proof}

\begin{proposition}\label{prop:strict}
Let $F:M\times[0,T)\to\R^{n+q}$ be a smooth mean curvature flow of a closed connected immersion with $A(\cdot,0)\in\cC_i$. Then $A(\cdot,t)\in\Int\cC_i$ and $H(\cdot,t)>0$ for every $t\in(0,T)$.
\end{proposition}

\begin{proof}
By Corollary~\ref{cor:preserved}, $A\in\cC_i$ throughout the flow. Suppose that $A(p_0,t_0)\in\partial\cC_i$ at some $(p_0,t_0)\in M\times(0,T)$. Since $M$ is connected, \eqref{eq:preservation-estimate} and Proposition~\ref{prop:maximum-equality} give
\[
 \min_{|u|=1}\Phi_u^{i,0}(A)=0 \quad\text{on }M\times(0,t_0].
\]
Fix $(p,t)\in M\times(0,t_0)$ and a unit $u$ such that $\Phi_u^{i,0}(A(p,t))=0$. Extend $u$ by parallel transport in time at $p$, and then along the radial geodesics from $p$ at each time. This gives $\nabla_tu=\nabla u=\Delta u=0$ at $(p,t)$. Since $\Phi_u^{i,0}(A)$ is nonnegative and vanishes at $(p,t)$, the chain rule, \eqref{eq:evolution}, Lemma~\ref{lem:codazzi}, and \eqref{eq:preservation-estimate} with $Z=0$ give
\begin{align*}
 0&\ge(\partial_t-\Delta)\Phi_u^{i,0}(A)\\
 &=D\Phi_u^{i,0}(A)[\Q(A)] -\sum_jD^2\Phi_u^{i,0}(A)[\nabla_jA,\nabla_jA]\\
 &\ge D\Phi_u^{i,0}(A)[\Q(A)]+\chi_n|\nabla A|^2\ge0.
\end{align*}
Thus $\nabla A=0$ on $M\times(0,t_0)$, and $D\Phi_u^{i,0}(A)[\Q(A)]=0$ for every unit $u$ with $\Phi_u^{i,0}(A)=0$.

We next show that $A(v,v)\ne0$ for every unit tangent vector $v$ at every time $t\in(0,t_0)$. Suppose instead that $A_p(v,v)=0$, and let $\gamma:\R\to M$ be the unit-speed $g(t)$-geodesic with $\gamma(0)=p$ and $\dot\gamma(0)=v$. Since $M$ is closed, $\gamma$ is defined for all $s\in\R$. The identities $\nabla A=0$ and $\nabla_{\dot\gamma}\dot\gamma=0$ give
\[
 \nabla^\perp_{\dot\gamma}\bigl(A(\dot\gamma,\dot\gamma)\bigr) =(\nabla_{\dot\gamma}A)(\dot\gamma,\dot\gamma) +2A(\nabla_{\dot\gamma}\dot\gamma,\dot\gamma)=0.
\]
Hence $|A(\dot\gamma(s),\dot\gamma(s))|=|A_p(v,v)|=0$ for all $s\in\R$. The Gauss formula gives
\[
 \frac{d^2}{ds^2}F(\gamma(s),t) =A(\dot\gamma(s),\dot\gamma(s))=0,\qquad \left|\frac d{ds}F(\gamma(s),t)\right|=1.
\]
Thus the compact set $F(M,t)$ contains the unit-speed straight line $F(\gamma(s),t)$, a contradiction. We have proved $A(v,v)\ne0$. It follows from \eqref{eq:coarse} that $H>0$ on $M\times(0,t_0)$, and the equality statement in Proposition~\ref{prop:reaction-first} or~\ref{prop:reaction-second} gives $\widehat A=0$.

Choose an orthonormal eigenbasis $e_1,\ldots,e_n$ of $A^{\bH}/H$. Then $A(e_j,e_k)=\lambda_j\delta_{jk}\bH/H$, with $\lambda_j\ne0$. The endomorphism $A^{\bH}/H$ is parallel. Commuting its covariant derivatives and using the Gauss equation gives
\[
 \lambda_j\lambda_k(\lambda_j-\lambda_k)=0\qquad(j\ne k).
\]
Since each $\lambda_j$ is nonzero, all $\lambda_j$ agree. Hence $A^\circ=0$ and $\Phi_u^{i,0}(A)\ge H^2/(n-1)>0$, a contradiction. Therefore $A\in\Int\cC_i$ at positive times. If $H(p,t)=0$ for some $t>0$, then \eqref{eq:coarse} gives $A(p,t)=0\in\partial\cC_i$, again a contradiction.
\end{proof}

\begin{corollary}\label{cor:margin}
Let $F:M\times[0,T)\to\R^{n+q}$ be a smooth mean curvature flow of a closed immersion with $A(\cdot,0)\in\cC_i$. For every $\tau\in(0,T)$, there is $\eta\in(0,1/2]$ such that $A(\cdot,t)\in\cC_i(\eta)$ for all $t\in[\tau,T)$.
\end{corollary}

\begin{proof}
Fix $\tau\in(0,T)$. Proposition~\ref{prop:strict}, applied to each connected component of $M$, gives $H>0$ and $\Phi_u^{i,0}(A)>0$ at time $\tau$ for every $p\in M$ and every unit $u$. By compactness, there is $\eta\in(0,1/2]$ such that
\[
 \frac{\Phi_u^{i,\eta}(A)}{H^2} \ge\frac{\Phi_u^{i,0}(A)}{H^2} -\frac{2(n+1)\eta}{(n-1)(n-2)}>0 \qquad\text{at }t=\tau.
\]
Thus $A\in\cC_i(\eta)$ at time $\tau$, and Corollary~\ref{cor:preserved} gives the conclusion.
\end{proof}

\subsection{Derivative estimates}

\begin{proposition}\label{prop:derivatives}
Let $F:M\times[0,T)\to\R^{n+q}$ be a smooth mean curvature flow of a closed immersion with $A(\cdot,0)\in\cC_i$ and $H(\cdot,0)>0$. There are constants $C_1,C_2,C$, depending only on the initial immersion, such that on $M\times[0,T)$,
\begin{equation}\label{eq:derivatives}
 |\nabla A|\le C_1H^2,\qquad |\nabla^2A|\le C_2H^3,\qquad |\nabla H^{-1}|+|\partial_tH^{-2}|\le C.
\end{equation}
\end{proposition}

\begin{proof}
Proposition~\ref{prop:strict}, applied to each connected component, gives $H>0$ throughout the flow. Set $G=2H^2-(n-1)|A|^2$. By \eqref{eq:coarse}, $0<|A|^2\le G\le2H^2$. The evolution equations and \eqref{eq:codazzi-trace} give
\begin{align*}
 (\partial_t-\Delta)G &\ge2(n-1)|\nabla A|^2-4|\nabla\bH|^2-C|A|^4 \ge\frac{2(n-7)}3|\nabla A|^2-CG^2.
\end{align*}
Set $\mathcal L=\partial_t-\Delta-2\nabla\log G\cdot\nabla$ and $f=|\nabla A|^2/G^2$. The evolution equations in \cite[Proposition~9]{AB} and the quotient rule give
\[
 \mathcal Lf \le-2\left|\nabla\!\left(\frac{\nabla A}{G}\right)\right|^2 +Gf\left(C-\frac{4(n-7)}3f\right).
\]
By the maximum principle, $f\le\max\{\max_M f(\cdot,0),3C/[4(n-7)]\}$. In particular, $|\nabla\log G|^2\le C|\nabla A|^2/G\le CG$. Set $h=|\nabla^2A|^2/G^3$ and choose $K>0$ sufficiently large. Using $|\nabla\log G|^2\le CG$, the evolution equations, and Young's inequality, we obtain
\begin{align*}
 \mathcal L(h+Kf) &\le-\frac{2|\nabla^3A|^2}{G^3} +\frac{8|\nabla\log G|\,|\nabla^3A|\,|\nabla^2A|}{G^3}-2K\left|\nabla\!\left(\frac{\nabla A}{G}\right)\right|^2+CGh+C(1+K)G\\
 &\le G(C_K-h-Kf).
\end{align*}
Applying the maximum principle, we obtain
\[
 0\le h\le h+Kf \le\max\left\{C_K,\max_M(h+Kf)(\cdot,0)\right\}.
\]
Since $G\le2H^2$, this proves the first two inequalities in \eqref{eq:derivatives}. The identities $\nabla^\perp\bH=\tr(\nabla A)$ and $\Delta^\perp\bH=\tr(\Delta A)$, together with \eqref{eq:mean-evolution}, give
\begin{align*}
 |\nabla H^{-1}| &\le\frac{|\nabla^\perp\bH|}{H^2} \le\frac{\sqrt n\,|\nabla A|}{H^2}\le C,\\
 |\partial_tH^{-2}| &\le\frac2{H^3}\bigl(|\Delta^\perp\bH|+|A|^2H\bigr) \le\frac{2n|\nabla^2A|}{H^3}+\frac{2|A|^2}{H^2}\le C.
\end{align*}
\end{proof}

\begin{corollary}\label{cor:local}
Under the hypotheses of Proposition~\ref{prop:derivatives}, there are constants $r,\theta>0$ with the following property. If $(p,t)\in M\times[0,T)$, $\Lambda=H(p,t)$, and $t\ge\theta\Lambda^{-2}$, then
\[
 \tfrac12\Lambda\le H\le2\Lambda \quad\text{on }B_{g(t)}(p,r/\Lambda)\times [t-\theta/\Lambda^2,t].
\]
For each $k\ge0$,
\[
 |\nabla^kA|\le C_k\Lambda^{k+1} \quad\text{on }B_{g(t)}\!\left(p,\frac r{2\Lambda}\right) \times\left[t-\frac\theta{2\Lambda^2},t\right].
\]
The constants depend only on the initial immersion and $k$.
\end{corollary}

\begin{proof}
Let $C$ be the constant in Proposition~\ref{prop:derivatives}, increased if necessary, and choose $r,\theta>0$ such that $Cr\le1/4$ and $C\theta\le1/4$. Fix $x\in B_{g(t)}(p,r/\Lambda)$ and $s\in[t-\theta/\Lambda^2,t]$, and let $\gamma$ be a minimizing $g(t)$-geodesic from $p$ to $x$. Proposition~\ref{prop:derivatives} gives
\begin{align*}
 \left|\frac{\Lambda}{H(x,t)}-1\right| &\le\Lambda\int_\gamma|\nabla H^{-1}(\cdot,t)|\,d\ell \le C\Lambda\,d_{g(t)}(p,x)\le Cr,\\
 \Lambda^2\left|H(x,s)^{-2}-H(x,t)^{-2}\right| &\le\Lambda^2\int_s^t |\partial_\sigma H(x,\sigma)^{-2}|\,d\sigma \le C\Lambda^2(t-s)\le C\theta.
\end{align*}
Consequently,
\[
 \frac14\le(1-Cr)^2-C\theta \le\Lambda^2H(x,s)^{-2} \le(1+Cr)^2+C\theta\le4,
\]
which proves $\Lambda/2\le H(x,s)\le2\Lambda$. By \eqref{eq:coarse}, $|A|^2\le8\Lambda^2/n$ on $B_{g(t)}(p,r/\Lambda)\times[t-\theta/\Lambda^2,t]$. Since $|A^{\bH}|\le H|A|\le C\Lambda^2$, the metric evolution $\partial_sg=-2A^{\bH}$ gives
\[
 e^{-C\theta}g(t)\le g(s)\le e^{C\theta}g(t) \quad\text{on }B_{g(t)}(p,r/\Lambda).
\]
After rescaling space by $\Lambda$ and time by $\Lambda^2$, the curvature and metric bounds above allow us to apply the interior derivative estimates for mean curvature flow on $B_{g(t)}(p,r/(2\Lambda))\times[t-\theta/(2\Lambda^2),t]$; see \cite[proof of Lemma~3.4]{NaffCanonical}. This proves the remaining estimates.
\end{proof}



\section{Deforming the second fundamental form}\label{sec:deformation}
Throughout this section, we assume that $n\ge10$, $i\in\{1,2\}$, $0<\eta\le1/2$, and $H>0$. Following the construction of the linear transformations $\ell_{a,b}$ in Ricci flow \cite{BW}, we define
\begin{equation}\label{eq:theta}
 \Theta_s(A)=\widehat A+s\frac{A^{\bH}}{H^2}\otimes\bH \qquad(s>0).
\end{equation}
Since $\tr\Theta_s(A)=s\bH$ and $\widehat{\Theta_s(A)}=\widehat A$, we have
\[
 \Theta_s^{-1}(A)=\Theta_{1/s}(A) =\widehat A+\frac{A^{\bH}}{sH^2}\otimes\bH.
\]
For $s\ge1$, we define the family of curvature cones
\[
 \cC_{i,\eta}(s)=\Theta_s\bigl(\cC_i(\eta)\cap\{H>0\}\bigr)=\{A:H>0,\ \Theta_s^{-1}(A)\in\cC_i(\eta)\}.
\]
As before, $u$ is a unit vector in $V$ if $i=1$, and a unit two-form in $\Lambda^2V$ if $i=2$.

\begin{lemma}\label{lem:family}
Let $s\ge1$. The set $\cC_{i,\eta}(s)$ is invariant under the natural $O(V)\times O(N)$ action and positive scalar multiplication, and satisfies $s^2|\widehat A|^2\le2H^2/n$ for every $A\in\cC_{i,\eta}(s)$.
\end{lemma}
\begin{proof}
Invariance follows from \eqref{eq:theta}. For $A\in\cC_{i,\eta}(s)$, we have $\Theta_s^{-1}(A)\in\cC_i(\eta)$ and $\tr\Theta_s^{-1}(A)=\bH/s$. By \eqref{eq:coarse}, $s^2|\widehat A|^2 \le s^2|\Theta_s^{-1}(A)|^2\le\frac{2H^2}{n}$.
\end{proof}

\Needspace{10\baselineskip}
\begin{lemma}\label{lem:deformation-calculus}
Let $s\ge1$ and $A\in\Sym^2(V)\otimes N$ with $H>0$. Fix an orthonormal normal frame $\nu_1,\ldots,\nu_q$ with $\nu_1=\bH/H$, so that $A_1=A^{\bH}/H$ and $\widehat A=\sum_{\alpha>1}A_\alpha^\circ\otimes\nu_\alpha$. Then the following identities hold for every unit $u$.
\begin{enumerate}[label=(\roman*)]
\item The components of $\Q(\Theta_s^{-1}(A))$ are given by
\begin{equation}\label{eq:deformed-q}
\begin{aligned}
 \Q(\Theta_s^{-1}(A))_1 &=s^{-1}\Q(A)_1+(s^{-3}-s^{-1})|A_1|^2A_1,\\
 \Q(\Theta_s^{-1}(A))_\alpha &=s^{-2}\bigl(\tr(A_\alpha A_1)A_1+A_1^2A_\alpha +A_\alpha A_1^2-2A_1A_\alpha A_1\bigr) +\Q(\widehat A)_\alpha\qquad(\alpha>1).
\end{aligned}
\end{equation}

\item We have
\begin{equation}\label{eq:deformed-reaction-identity}
\begin{aligned}
 &D\left(\frac{\Phi_u^{i,\eta}\circ\Theta_s^{-1}}{H}\right)(A)[\Q(A)]\\
 &=|A_1|^2\frac{\Phi_u^{i,\eta}(\Theta_s^{-1}(A))}{H}\\
 &\quad+\frac1H\left( D\Phi_u^{i,\eta}(\Theta_s^{-1}(A))[\Q(\Theta_s^{-1}(A))] -2s^{-2}|A_1|^2\Phi_u^{i,\eta}(\Theta_s^{-1}(A))\right)\\
 &\quad+\frac{1-s^{-2}}{H}D\Phi_u^{i,\eta}(\widehat A) [A_1^2\widehat A+\widehat A A_1^2 -2A_1\widehat A A_1-|A_1|^2\widehat A].
\end{aligned}
\end{equation}

\item Let $Z\in\Sym^3(V)\otimes N$ and $Z_j=Z(e_j,\cdot,\cdot)$. Write $(\tr Z_\alpha)_j=\tr Z_{\alpha j}$, and define the traceless tensors $\widetilde Z_j=\sum_\alpha\widetilde Z_{\alpha j}\otimes\nu_\alpha$ by
\[
\begin{aligned}
 \widetilde Z_{1j} &=Z_{1j}^\circ-\frac{\tr Z_{1j}}{H}A_1^\circ -(s^2-1)\sum_{\alpha>1}\frac{\tr Z_{\alpha j}}{H}A_\alpha^\circ,\\
 \widetilde Z_{\alpha j} &=sZ_{\alpha j}^\circ-s\frac{\tr Z_{1j}}{H}A_\alpha^\circ -(s-s^{-1})\frac{\tr Z_{\alpha j}}{H}A_1^\circ \qquad(\alpha>1).
\end{aligned}
\]
Then
\begin{equation}\label{eq:deformed-second-identity}
\begin{aligned}
 &-\frac12\sum_jD^2\left(\frac{\Phi_u^{i,\eta}\circ\Theta_s^{-1}}{H}\right)(A)[Z_j,Z_j]\\
 &=-\frac1{s^2H}\sum_j\Phi_u^{i,\eta}(\widetilde Z_j)\\
 &\quad-\frac1{H^3}\left(\frac12\Phi_u^{i,\eta}(\Theta_s^{-1}(A)) +s^{-4}[-\Phi_u^{i,\eta}(A_1^\circ\otimes\nu_1)] -\Phi_u^{i,\eta}(\widehat A)\right) \sum_{\alpha>1}|\tr Z_\alpha|^2\\
 &\quad+\frac{s^2-1}{H^3}\sum_j\Phi_u^{i,\eta}\left( \sum_{\alpha>1}(\tr Z_{\alpha j})A_\alpha^\circ\otimes\nu_1\right).
\end{aligned}
\end{equation}
\end{enumerate}
Products of a matrix and $\widehat A$ are taken in each normal component.
\end{lemma}

\begin{proof}
In the chosen normal frame, the components of $\Theta_s^{-1}(A)$ are $s^{-1}A_1,A_2,\ldots,A_q$. Substituting these in \eqref{eq:reaction} proves (i). By \eqref{eq:phi1} and \eqref{eq:phi2}, we also have
\begin{equation}\label{eq:deformed-splitting}
\begin{aligned}
 \Phi_u^{i,\eta}(\Theta_s^{-1}(A)) &=s^{-2}\Phi_u^{i,\eta}\left(\frac{A^{\bH}}{H^2}\otimes\bH\right) +\Phi_u^{i,\eta}(\widehat A)\\
 &=s^{-2}\Phi_u^{i,\eta}(A)+(1-s^{-2})\Phi_u^{i,\eta}(\widehat A).
\end{aligned}
\end{equation}
To prove (ii) and (iii), we may assume by homogeneity that $H=1$. In the chosen normal frame, $\nu_1=\bH$. Taking the trace of \eqref{eq:reaction} gives
\[
\begin{aligned}
 \tr\Q(A)_\alpha &=\sum_\beta\tr(A_\alpha A_\beta)\tr A_\beta =\tr(A_1A_\alpha),\\
 DH(A)[\Q(A)]&=\ip{\bH}{\tr\Q(A)}=|A_1|^2.
\end{aligned}
\]
Since $\widehat A=A-(A^{\bH}/H^2)\otimes\bH$, we have
\[
\begin{aligned}
 D\widehat A(A)[\Q(A)] &=\sum_{\alpha>1}\bigl(\Q(A)_\alpha-\tr(A_1A_\alpha)A_1\bigr) \otimes\nu_\alpha -\left(\sum_{\alpha>1}\tr(A_1A_\alpha)A_\alpha\right)\otimes\nu_1\\
 &=\Q(\widehat A)+A_1^2\widehat A+\widehat A A_1^2-2A_1\widehat A A_1 -\left(\sum_{\alpha>1}\tr(A_1A_\alpha)A_\alpha\right)\otimes\nu_1.
\end{aligned}
\]
Since $\widehat A_1=0$, \eqref{eq:phi1} and \eqref{eq:phi2} give
\[
 \Phi_u^{i,\eta}(\widehat A+rU\otimes\nu_1) =\Phi_u^{i,\eta}(\widehat A)+r^2\Phi_u^{i,\eta}(U\otimes\nu_1) \qquad(U\in\Sym^2(V),\ r\in\R).
\]
Hence $D\Phi_u^{i,\eta}(\widehat A)[U\otimes\nu_1]=0$, and \eqref{eq:deformed-splitting} implies
\[
\begin{aligned}
 &D(\Phi_u^{i,\eta}\circ\Theta_s^{-1})(A)[\Q(A)]\\
 &=s^{-2}D\Phi_u^{i,\eta}(A)[\Q(A)] +(1-s^{-2})D\Phi_u^{i,\eta}(\widehat A) [D\widehat A(A)[\Q(A)]]\\
 &=s^{-2}D\Phi_u^{i,\eta}(A)[\Q(A)] +(1-s^{-2})D\Phi_u^{i,\eta}(\widehat A) [\Q(\widehat A)+A_1^2\widehat A+\widehat A A_1^2-2A_1\widehat A A_1].
\end{aligned}
\]
On the other hand, \eqref{eq:deformed-q} gives
\[
\begin{aligned}
 &D\Phi_u^{i,\eta}(\Theta_s^{-1}(A))[\Q(\Theta_s^{-1}(A))]\\
 &=s^{-2}D\Phi_u^{i,\eta}(A)[\Q(A)] +2(s^{-4}-s^{-2})|A_1|^2\Phi_u^{i,\eta}(A_1\otimes\nu_1)\\
 &\quad+(1-s^{-2})D\Phi_u^{i,\eta}(\widehat A)[\Q(\widehat A)].
\end{aligned}
\]
Since $D\Phi_u^{i,\eta}(\widehat A)[\widehat A] =2\Phi_u^{i,\eta}(\widehat A)$ and $DH(A)[\Q(A)]=|A_1|^2$, the last two identities imply
\[
\begin{aligned}
 &D\left(\frac{\Phi_u^{i,\eta}\circ\Theta_s^{-1}}{H}\right)(A)[\Q(A)]\\
 &=D(\Phi_u^{i,\eta}\circ\Theta_s^{-1})(A)[\Q(A)] -|A_1|^2\Phi_u^{i,\eta}(\Theta_s^{-1}(A))\\
 &=|A_1|^2\Phi_u^{i,\eta}(\Theta_s^{-1}(A)) +D\Phi_u^{i,\eta}(\Theta_s^{-1}(A))[\Q(\Theta_s^{-1}(A))] -2s^{-2}|A_1|^2\Phi_u^{i,\eta}(\Theta_s^{-1}(A))\\
 &\quad+(1-s^{-2})D\Phi_u^{i,\eta}(\widehat A) [A_1^2\widehat A+\widehat A A_1^2 -2A_1\widehat A A_1-|A_1|^2\widehat A].
\end{aligned}
\]
This proves (ii).

To prove (iii), fix $j$ and let $Z_j\in\Sym^2(V)\otimes N$. Since $(\Phi_u^{i,\eta}\circ\Theta_s^{-1})/H$ is homogeneous of degree one,
\[
\begin{aligned}
 D^2\left(\frac{\Phi_u^{i,\eta}\circ\Theta_s^{-1}}H\right)(A)[A,Z_j]&=0,\\
 D^2\left(\frac{\Phi_u^{i,\eta}\circ\Theta_s^{-1}}H\right)(A) [Z_j-(\tr Z_{1j})A,Z_j-(\tr Z_{1j})A] &=D^2\left(\frac{\Phi_u^{i,\eta}\circ\Theta_s^{-1}}H\right)(A)[Z_j,Z_j].
\end{aligned}
\]
The definition of $\widetilde Z_j$ is linear in $Z_j$ and gives zero when $Z_j=A$. Hence replacing $Z_j$ by $Z_j-(\tr Z_{1j})A$ leaves $\widetilde Z_j$ unchanged. Moreover, $\tr A_1=1$ and $\tr A_\alpha=0$ for $\alpha>1$ give
\[
 \tr\bigl(Z_{1j}-(\tr Z_{1j})A_1\bigr)=0,\qquad \tr\bigl(Z_{\alpha j}-(\tr Z_{1j})A_\alpha\bigr) =\tr Z_{\alpha j}\quad(\alpha>1).
\]
Thus both sides of \eqref{eq:deformed-second-identity} are unchanged, and we may assume that $\tr Z_{1j}=0$. For sufficiently small $r$, we have
\[
\begin{aligned}
 |\bH+r\tr Z_j|^2 &=1+r^2\sum_{\alpha>1}(\tr Z_{\alpha j})^2,\\
 \sum_\alpha\langle\bH+r\tr Z_j,\nu_\alpha\rangle (A_\alpha+rZ_{\alpha j}) &=A_1+r\left(Z_{1j} +\sum_{\alpha>1}(\tr Z_{\alpha j})A_\alpha^\circ\right) +r^2\sum_{\alpha>1}(\tr Z_{\alpha j})Z_{\alpha j}.
\end{aligned}
\]
Applying \eqref{eq:deformed-splitting} to $A+rZ_j$ and using invariance under orthogonal changes of the normal frame, we obtain
\[
\begin{aligned}
 \frac{\Phi_u^{i,\eta}(\Theta_s^{-1}(A+rZ_j))}{|\bH+r\tr Z_j|} &=\frac{\Phi_u^{i,\eta}(A+rZ_j)}{|\bH+r\tr Z_j|}\\
 &\quad+\frac{s^{-2}-1}{|\bH+r\tr Z_j|^3} \Phi_u^{i,\eta}\left( \sum_\alpha\langle\bH+r\tr Z_j,\nu_\alpha\rangle (A_\alpha+rZ_{\alpha j})\otimes\nu_1\right).
\end{aligned}
\]
Comparing the terms of order $r^2$ and using Lemma~\ref{lem:phi-derivatives}(iii), we obtain
\begin{align*}
 &-\frac12D^2\left(\frac{\Phi_u^{i,\eta}\circ\Theta_s^{-1}}{H}\right)(A)[Z_j,Z_j]\\
 &=-\Phi_u^{i,\eta}(Z_j) +(1-s^{-2})\Phi_u^{i,\eta}\left( \left[Z_{1j}+\sum_{\alpha>1}(\tr Z_{\alpha j})A_\alpha^\circ\right] \otimes\nu_1\right)\\*
 &\quad+(1-s^{-2})D\Phi_u^{i,\eta}(A_1\otimes\nu_1) \left[\sum_{\alpha>1}(\tr Z_{\alpha j})Z_{\alpha j}\otimes\nu_1\right]\\*
 &\quad+\left[\frac12\Phi_u^{i,\eta}(\Theta_s^{-1}(A)) -(1-s^{-2})\Phi_u^{i,\eta}(A_1\otimes\nu_1)\right] \sum_{\alpha>1}(\tr Z_{\alpha j})^2.
\end{align*}
Since $\tr Z_{1j}=0$, $\tr A_1=1$, and $\tr A_\alpha=0$ for $\alpha>1$, completing squares gives
\begin{align*}
 &-\frac12D^2\left(\frac{\Phi_u^{i,\eta}\circ\Theta_s^{-1}}{H}\right)(A)[Z_j,Z_j]\\
 &=-s^{-2}\Phi_u^{i,\eta}\left( \left[Z_{1j}-(s^2-1)\sum_{\alpha>1}(\tr Z_{\alpha j})A_\alpha^\circ\right] \otimes\nu_1\right)\\*
 &\quad-s^{-2}\sum_{\alpha>1}\Phi_u^{i,\eta}\left( \left[sZ_{\alpha j}^\circ-(s-s^{-1})(\tr Z_{\alpha j})A_1^\circ\right] \otimes\nu_\alpha\right)\\*
 &\quad+(s^2-1)\Phi_u^{i,\eta}\left( \sum_{\alpha>1}(\tr Z_{\alpha j})A_\alpha^\circ\otimes\nu_1\right)\\*
 &\quad+\left[\frac12\Phi_u^{i,\eta}(\Theta_s^{-1}(A)) -s^{-2}\Phi_u^{i,\eta}(A_1\otimes\nu_1) +s^{-4}\Phi_u^{i,\eta}(A_1^\circ\otimes\nu_1)\right] \sum_{\alpha>1}(\tr Z_{\alpha j})^2.
\end{align*}
The first two terms equal $-s^{-2}\Phi_u^{i,\eta}(\widetilde Z_j)$ by definition. For the last term, \eqref{eq:deformed-splitting} gives
\[
\begin{aligned}
 &\frac12\Phi_u^{i,\eta}(\Theta_s^{-1}(A)) -s^{-2}\Phi_u^{i,\eta}(A_1\otimes\nu_1) +s^{-4}\Phi_u^{i,\eta}(A_1^\circ\otimes\nu_1)\\
 &=-\frac12\Phi_u^{i,\eta}(\Theta_s^{-1}(A)) -s^{-4}[-\Phi_u^{i,\eta}(A_1^\circ\otimes\nu_1)] +\Phi_u^{i,\eta}(\widehat A).
\end{aligned}
\]
Summing over $j$ proves (iii) when $H=1$. The general identity follows by homogeneity.
\end{proof}

\Needspace{12\baselineskip}
\begin{proposition}\label{prop:deformed-reaction}
Let $s\ge1$ and $A\in\cC_{i,\eta/2}(s)$, and let $u$ be unit such that $\Phi_u^{i,\eta}(\Theta_s^{-1}(A)) =\min_{|v|=1}\Phi_v^{i,\eta}(\Theta_s^{-1}(A))$. Then
\[
\begin{aligned}
 D\left(\frac{\Phi_u^{i,\eta}\circ\Theta_s^{-1}}{H}\right)(A)[\Q(A)] &\ge |A|^2\min\left\{\frac{\Phi_u^{i,\eta}(\Theta_s^{-1}(A))}{H},0\right\}+\frac{\eta H}{n}[-\Phi_u^{i,\eta}(\widehat A)].
\end{aligned}
\]
\end{proposition}
\begin{proof}
By homogeneity, we may assume that $H=1$. Fix an orthonormal normal frame $\nu_1,\ldots,\nu_q$ with $\nu_1=\bH$, so that $A_1=A^{\bH}$. On the space of traceless tensors, Lemma~\ref{lem:scalar} bounds the eigenvalues of $-\Phi_u^{i,\eta}$ between $1$ and $n+3$. Taking the norm of $D\Phi_u^{i,\eta}(\widehat A)$ on this space and using $r^2+n+3\le(n+4)r$ for $1\le r\le n+3$, Lemma~\ref{lem:phi-derivatives}(iii) gives
\[
\begin{aligned}
 |D\Phi_u^{i,\eta}(\widehat A)|\,|\widehat A| &\le\frac{|D\Phi_u^{i,\eta}(\widehat A)|^2 +4(n+3)|\widehat A|^2}{4\sqrt{n+3}}\\
 &\le\frac{n+4}{\sqrt{n+3}}[-\Phi_u^{i,\eta}(\widehat A)] \le\frac{n-2}{2}[-\Phi_u^{i,\eta}(\widehat A)].
\end{aligned}
\]
Choose an orthonormal eigenbasis of $A_1^\circ$ with eigenvalues $\lambda_1\le\cdots\le\lambda_n$. Then
\[
 (A_1^2\widehat A_\alpha+\widehat A_\alpha A_1^2 -2A_1\widehat A_\alpha A_1)_{jk} =(\lambda_j-\lambda_k)^2(\widehat A_\alpha)_{jk}.
\]
We have $|A_1|^2=1/n+|A_1^\circ|^2$, $(\lambda_n-\lambda_1)^2\le2|A_1^\circ|^2$, and $(\lambda_n-\lambda_1)^2\le4\max\{\lambda_1^2,\lambda_n^2\}$. Moreover, \eqref{eq:shape-square-bound} applies to $\Theta_s^{-1}(A)$, whose trace is $\bH/s$. Using these inequalities, $|(\lambda_j-\lambda_k)^2-\frac12(\lambda_n-\lambda_1)^2| \le\frac12(\lambda_n-\lambda_1)^2$, and $D\Phi_u^{i,\eta}(\widehat A)[\widehat A] =2\Phi_u^{i,\eta}(\widehat A)$, we obtain
\begin{equation}\label{eq:mean-product-polarization}
\begin{aligned}
 &D\Phi_u^{i,\eta}(\widehat A) [A_1^2\widehat A+\widehat A A_1^2-2A_1\widehat A A_1-|A_1|^2\widehat A]\\
 &\ge\left(2|A_1|^2-(\lambda_n-\lambda_1)^2\right) [-\Phi_u^{i,\eta}(\widehat A)] -\frac{(\lambda_n-\lambda_1)^2}{2} |D\Phi_u^{i,\eta}(\widehat A)|\,|\widehat A|\\
 &\ge\left(2|A_1|^2-\frac{n+2}{4}(\lambda_n-\lambda_1)^2\right) [-\Phi_u^{i,\eta}(\widehat A)]\\
 &\ge\left(\frac2n-(n-2)\max\{\lambda_1^2,\lambda_n^2\}\right) [-\Phi_u^{i,\eta}(\widehat A)] \ge\frac\eta n[-\Phi_u^{i,\eta}(\widehat A)].
\end{aligned}
\end{equation}
Since $\Theta_s^{-1}(A)\in\cC_i(\eta/2)$, $\tr\Theta_s^{-1}(A)=\bH/s$, and $\widehat{\Theta_s^{-1}(A)}=\widehat A$, \eqref{eq:first-sharp-reaction} and \eqref{eq:second-sharp-reaction} give, respectively,
\[
\begin{aligned}
 &D\Phi_u^{i,\eta}(\Theta_s^{-1}(A))[\Q(\Theta_s^{-1}(A))] -2s^{-2}|A_1|^2\Phi_u^{i,\eta}(\Theta_s^{-1}(A))\\
 &\ge\frac\eta{ns^2}[-\Phi_u^{i,\eta}(\widehat A)].
\end{aligned}
\]
Substituting this inequality and \eqref{eq:mean-product-polarization} in \eqref{eq:deformed-reaction-identity}, and using $|A_1|^2\le|A|^2$, we obtain
\[
\begin{aligned}
 &D\left(\frac{\Phi_u^{i,\eta}\circ\Theta_s^{-1}}H\right)(A)[\Q(A)]\\
 &\ge |A_1|^2\Phi_u^{i,\eta}(\Theta_s^{-1}(A)) +\frac\eta{ns^2}[-\Phi_u^{i,\eta}(\widehat A)] +\frac{\eta(1-s^{-2})}n[-\Phi_u^{i,\eta}(\widehat A)]\\
 &\ge |A|^2\min\{\Phi_u^{i,\eta}(\Theta_s^{-1}(A)),0\} +\frac\eta n[-\Phi_u^{i,\eta}(\widehat A)].
\end{aligned}
\]
The assertion follows by homogeneity.
\end{proof}

\begin{samepage}
\begin{proposition}\label{prop:deformed-hessian}
Let $s\ge1$ and $A\in\cC_{1,\eta/2}(s)$. Then
\[
 -\sum_jD^2\left(\frac{\Phi_u^{1,\eta}\circ\Theta_s^{-1}}{H}\right)(A)[Z_j,Z_j] \ge\frac{|\tr\langle Z,\bH\rangle|^2}{4(n-1)s^2H^3}
\]
for every unit vector $u\in V$ and every $Z\in\Sym^3(V)\otimes N$. For $A\in\cC_{2,\eta/2}(s)$ and unit $u\in\Lambda^2V$, the same estimate holds with $\Phi_u^{1,\eta}$ replaced by $\frac{n-2}{2(n+1)}\Phi_u^{2,\eta}$.
\end{proposition}
\end{samepage}

\begin{proof}
By homogeneity, we may assume that $H=1$. Fix an orthonormal basis $e_1,\ldots,e_n$ of $V$ and an orthonormal normal frame $\nu_1,\ldots,\nu_q$ with $\nu_1=\bH$. Let $Z_j=Z(e_j,\cdot,\cdot)$ and let $\widetilde Z_j$ be as in Lemma~\ref{lem:deformation-calculus}(iii).

Suppose first that $i=1$. For the last term in \eqref{eq:deformed-second-identity}, Cauchy--Schwarz in the inner product of Lemma~\ref{lem:phi-derivatives}(iii) gives
\[
 -\Phi_u^{1,\eta}\left( \sum_{\alpha>1}(\tr Z_{\alpha j})A_\alpha^\circ\otimes\nu_1\right) \le[-\Phi_u^{1,\eta}(\widehat A)] \sum_{\alpha>1}|\tr Z_{\alpha j}|^2.
\]
Since $\Theta_s^{-1}(A)\in\cC_1(\eta/2)$ and $H=1$, \eqref{eq:phi1} gives
\[
\begin{aligned}
 &\frac12\Phi_u^{1,\eta}(\Theta_s^{-1}(A)) +s^2[-\Phi_u^{1,\eta}((\Theta_s^{-1}(A))^\circ)]\\
 &=\frac{1-\eta}{2(n-1)s^2} +\left(s^2-\frac12\right) [-\Phi_u^{1,\eta}((\Theta_s^{-1}(A))^\circ)]\\
 &\le\frac{1-\eta/2}{n-1}-\frac{\eta}{4(n-1)s^2} \le\frac1{n-1}.
\end{aligned}
\]
Substituting these inequalities in \eqref{eq:deformed-second-identity} and applying Lemma~\ref{lem:scalar}, we obtain
\begin{equation}\label{eq:deformed-hessian-reduction}
\begin{aligned}
 &-\frac12\sum_jD^2\left(\frac{\Phi_u^{1,\eta}\circ\Theta_s^{-1}}{H}\right)(A)[Z_j,Z_j]\\
 &\ge\frac1{s^2}\sum_j[-\Phi_u^{1,\eta}(\widetilde Z_j)] +(1-s^{-4})[-\Phi_u^{1,\eta}(A_1^\circ\otimes\nu_1)] \sum_{\alpha>1}|\tr Z_\alpha|^2 -\frac1{n-1}\sum_{\alpha>1}|\tr Z_\alpha|^2\\
 &\ge\frac1{s^2}\left(\sum_j|\widetilde Z_j|^2 +(s^2-s^{-2})|A_1^\circ|^2\sum_{\alpha>1}|\tr Z_\alpha|^2\right) -\frac1{n-1}\sum_{\alpha>1}|\tr Z_\alpha|^2.
\end{aligned}
\end{equation}
It is therefore enough to prove
\begin{equation}\label{eq:deformed-hessian-contraction}
 \sum_j|\widetilde Z_j|^2 +(s^2-s^{-2})|A_1^\circ|^2\sum_{\alpha>1}|\tr Z_\alpha|^2 \ge\frac1{8(n-1)}|\tr Z_1|^2 +\frac{s^2}{n-1}\sum_{\alpha>1}|\tr Z_\alpha|^2.
\end{equation}

Since $Z\in\Sym^3(V)\otimes N$, we have $\sum_jZ_{\alpha j}^\circ e_j=(n-1)\tr Z_\alpha/n$, and hence
\begin{equation}\label{eq:deformed-trace-identities}
\begin{aligned}
 \sum_j\widetilde Z_{1j}e_j &=\left(\frac{n-1}{n}\Id-A_1^\circ\right)\tr Z_1 -(s^2-1)\sum_{\alpha>1}A_\alpha^\circ\tr Z_\alpha,\\
 \sum_j\widetilde Z_{\alpha j}e_j &=\frac{s(n-1)}n\tr Z_\alpha-sA_\alpha^\circ\tr Z_1 -(s-s^{-1})A_1^\circ\tr Z_\alpha\quad(\alpha>1).
\end{aligned}
\end{equation}
Since $\tr\Theta_s^{-1}(A)=\bH/s$, applying \eqref{eq:shape-square-bound} to $\Theta_s^{-1}(A)$ gives
\begin{equation}\label{eq:deformed-matrix-bound}
 \frac{n^2}{(n-1)^2} \left((A_1^\circ)^2+s^2\sum_{\alpha>1}(A_\alpha^\circ)^2\right) \le\frac{2n}{(n-1)^2(n-2)}\Id\le\frac1{32}\Id.
\end{equation}
By the first identity in \eqref{eq:deformed-trace-identities}, the inequality $|v-w|^2\ge|v|^2/2-|w|^2$, and Cauchy--Schwarz with \eqref{eq:deformed-matrix-bound}, we obtain
\begin{equation}\label{eq:deformed-first-trace-bound}
\begin{aligned}
 \frac{n^2}{(n-1)^2} \left|\sum_j\widetilde Z_{1j}e_j\right|^2 &\ge\frac12|\tr Z_1|^2 -\left|\frac n{n-1}A_1^\circ\tr Z_1 +\frac{n(s^2-1)}{n-1} \sum_{\alpha>1}A_\alpha^\circ\tr Z_\alpha\right|^2\\
 &\ge\frac{15}{32}|\tr Z_1|^2 -\frac{s^2}{32}\sum_{\alpha>1}|\tr Z_\alpha|^2.
\end{aligned}
\end{equation}
For $\alpha>1$, the second identity in \eqref{eq:deformed-trace-identities} and Cauchy--Schwarz give
\begin{equation}\label{eq:deformed-other-trace-bound}
\begin{aligned}
 &s^2|\tr Z_\alpha|^2\\
 &=\left|\frac n{n-1}\sum_j\widetilde Z_{\alpha j}e_j +\frac{ns}{n-1}A_\alpha^\circ\tr Z_1 +\frac{n(s-s^{-1})}{n-1}A_1^\circ\tr Z_\alpha\right|^2\\
 &\le\left(\frac34+\frac1{16}+\frac18\right) \left[\frac{4n^2}{3(n-1)^2} \left|\sum_j\widetilde Z_{\alpha j}e_j\right|^2 +\frac{16n^2s^2}{(n-1)^2}|A_\alpha^\circ\tr Z_1|^2\right.\\
 &\qquad\left. +\frac{8n^2(s-s^{-1})^2}{(n-1)^2}|A_1^\circ\tr Z_\alpha|^2\right].
\end{aligned}
\end{equation}
Summing \eqref{eq:deformed-other-trace-bound} over $\alpha>1$, using $s^2\sum_{\alpha>1}(A_\alpha^\circ)^2\le(n-1)^2\Id/(32n^2)$ from \eqref{eq:deformed-matrix-bound}, and then applying \eqref{eq:deformed-first-trace-bound}, we obtain
\begin{equation}\label{eq:deformed-trace-estimate}
\begin{aligned}
 &\frac{4n^2}{3(n-1)^2} \sum_\alpha\left|\sum_j\widetilde Z_{\alpha j}e_j\right|^2 +\frac{8n^2(s-s^{-1})^2}{(n-1)^2} \sum_{\alpha>1}|A_1^\circ\tr Z_\alpha|^2\\
 &\ge\frac{4n^2}{3(n-1)^2} \left|\sum_j\widetilde Z_{1j}e_j\right|^2 +\frac{16s^2}{15}\sum_{\alpha>1}|\tr Z_\alpha|^2 -\frac12|\tr Z_1|^2\\
 &\ge\frac18|\tr Z_1|^2 +\frac{41s^2}{40}\sum_{\alpha>1}|\tr Z_\alpha|^2.
\end{aligned}
\end{equation}
For every $w\in V$, tracelessness of $\widetilde Z_{\alpha j}$ and Cauchy--Schwarz give
\[
\begin{aligned}
 \left|\left\langle\sum_j\widetilde Z_{\alpha j}e_j,w\right\rangle\right|^2 &=\left|\sum_j\tr\left[\widetilde Z_{\alpha j} \left(\frac{e_j\otimes w+w\otimes e_j}{2} -\frac{\langle e_j,w\rangle}{n}\Id\right)\right]\right|^2\\
 &\le\frac{(n-1)(n+2)}{2n}|w|^2 \sum_j|\widetilde Z_{\alpha j}|^2.
\end{aligned}
\]
Taking the supremum over $|w|=1$ and summing over $\alpha$ gives the first inequality below. The second follows from $\tr A_1^\circ=0$:
\begin{equation}\label{eq:deformed-contraction-bounds}
\begin{aligned}
 \sum_\alpha\left|\sum_j\widetilde Z_{\alpha j}e_j\right|^2 &\le\frac{(n-1)(n+2)}{2n}\sum_j|\widetilde Z_j|^2,\\
 \sum_{\alpha>1}|A_1^\circ\tr Z_\alpha|^2 &\le\frac{n-1}{n}|A_1^\circ|^2 \sum_{\alpha>1}|\tr Z_\alpha|^2.
\end{aligned}
\end{equation}
Since $2n(n+2)\le3(n-1)^2$, $8n\le(n-1)^2$, and $(s-s^{-1})^2\le s^2-s^{-2}$, \eqref{eq:deformed-contraction-bounds} and \eqref{eq:deformed-trace-estimate} give
\[
\begin{aligned}
 &\sum_j|\widetilde Z_j|^2 +(s^2-s^{-2})|A_1^\circ|^2\sum_{\alpha>1}|\tr Z_\alpha|^2\\
 &\ge\frac{2n(n+2)}{3(n-1)^2}\sum_j|\widetilde Z_j|^2 +\frac{8n(s-s^{-1})^2}{(n-1)^2}|A_1^\circ|^2 \sum_{\alpha>1}|\tr Z_\alpha|^2\\
 &\ge\frac{4n^2}{3(n-1)^3} \sum_\alpha\left|\sum_j\widetilde Z_{\alpha j}e_j\right|^2 +\frac{8n^2(s-s^{-1})^2}{(n-1)^3} \sum_{\alpha>1}|A_1^\circ\tr Z_\alpha|^2\\
 &\ge\frac1{8(n-1)}|\tr Z_1|^2 +\frac{s^2}{n-1}\sum_{\alpha>1}|\tr Z_\alpha|^2.
\end{aligned}
\]
Substituting \eqref{eq:deformed-hessian-contraction} in \eqref{eq:deformed-hessian-reduction}, we obtain
\[
\begin{aligned}
 &-\frac12\sum_jD^2\left(\frac{\Phi_u^{1,\eta}\circ\Theta_s^{-1}}{H}\right)(A)[Z_j,Z_j]\\
 &\ge\frac1{s^2}\left(\frac1{8(n-1)}|\tr Z_1|^2 +\frac{s^2}{n-1}\sum_{\alpha>1}|\tr Z_\alpha|^2\right) -\frac1{n-1}\sum_{\alpha>1}|\tr Z_\alpha|^2\\
 &=\frac1{8(n-1)s^2}|\tr Z_1|^2.
\end{aligned}
\]
For arbitrary $H>0$, we have $\tr\langle Z,\bH\rangle=H\tr Z_1$ in a normal frame with $\nu_1=\bH/H$. Homogeneity therefore gives the asserted estimate for $i=1$.

For $i=2$, replace $\Phi_u^{1,\eta}$ and $\Phi_u^{1,\eta/2}$ by $\frac{n-2}{2(n+1)}\Phi_u^{2,\eta}$ and $\frac{n-2}{2(n+1)}\Phi_u^{2,\eta/2}$. The terms involving $H^2$ are unchanged, and Lemma~\ref{lem:scalar} gives $-\frac{n-2}{2(n+1)}\Phi_u^{2,\eta}(X)\ge|X|^2$ for traceless $X$. The preceding calculation therefore gives the same lower bound for $\frac{n-2}{2(n+1)}\Phi_u^{2,\eta}$, and homogeneity proves the assertion for $i=2$.
\end{proof}

\section{Construction of pinching sets}\label{sec:pinching}
We now use Propositions~\ref{prop:deformed-reaction} and~\ref{prop:deformed-hessian} to construct pinching sets in the sense of Definition~\ref{def:pinching}, following \cite[Section~8.3]{BrendleBook}. Throughout this section, we assume that $n\ge10$, $i\in\{1,2\}$, $0<\eta\le1/2$, and $H>0$. As before, $u$ is a unit vector in $V$ if $i=1$, and a unit two-form in $\Lambda^2V$ if $i=2$. 

For $h\ge0$, the tensor $A+\frac{h}{H}\Id\otimes\bH$ has trace $(1+nh/H)\bH$ and traceless part $A^\circ$. Hence
\begin{equation}\label{eq:shifted-defining-functions}
\begin{aligned}
 \Phi_u^{1,\eta}\left(\Theta_s^{-1}\left( A+\frac{h}{H}\Id\otimes\bH\right)\right) &=\Phi_u^{1,\eta}(\Theta_s^{-1}(A)) +\frac{n(1-\eta)h(2H+nh)}{(n-1)s^2},\\
 \Phi_u^{2,\eta}\left(\Theta_s^{-1}\left( A+\frac{h}{H}\Id\otimes\bH\right)\right) &=\Phi_u^{2,\eta}(\Theta_s^{-1}(A)) +\frac{2n(n+1)(1-\eta)h(2H+nh)}{(n-1)(n-2)s^2}.
\end{aligned}
\end{equation}

\subsection{Preserved sets for the deformed cones}

\begin{lemma}\label{lem:shifted-estimate}
Let $s\ge1$ and $h>0$. Suppose that $H>0$ and $A+\frac{h}{H}\Id\otimes\bH\in\cC_{i,\eta}(s)$.
\begin{enumerate}[label=(\roman*)]
\item If $H\le h$, then
\[
 A+\frac{2h}{H}\Id\otimes\bH \in\cC_{i,\eta}\bigl((1+\eta/8)s\bigr).
\]
\begin{samepage}
\item If
\[
 A+\frac{2h}{H}\Id\otimes\bH \notin\operatorname{Int}\cC_{i,\eta}\bigl((1+\eta/8)s\bigr),
\]
then $A\in\cC_{i,\eta/2}\bigl((1+\eta/8)s\bigr)$.
\end{samepage}
\end{enumerate}
\end{lemma}
\begin{proof}
Suppose first that $i=1$ and $A+\frac{2h}{H}\Id\otimes\bH\notin\operatorname{Int} \cC_{1,\eta}((1+\eta/8)s)$. Then \eqref{eq:shifted-defining-functions} gives a unit vector $u$ such that
\[
 \Phi_u^{1,\eta}(\Theta_{(1+\eta/8)s}^{-1}(A)) +\frac{4n(1-\eta)h(H+nh)}{(n-1)(1+\eta/8)^2s^2}\le0.
\]
On the other hand, since $A+\frac{h}{H}\Id\otimes\bH\in\cC_{1,\eta}(s)$, \eqref{eq:deformed-splitting} and Lemma~\ref{lem:scalar} give
\begin{equation}\label{eq:shifted-cone-bound}
\begin{aligned}
 &\frac{(1-\eta)H^2}{(n-1)(1+\eta/8)^2s^2} -\Phi_v^{1,\eta}(\Theta_{(1+\eta/8)s}^{-1}(A))\\
 &=-\Phi_v^{1,\eta}((\Theta_{(1+\eta/8)s}^{-1}(A))^\circ)\\
 &\le-\Phi_v^{1,\eta}((\Theta_s^{-1}(A))^\circ)\\
 &\le\frac{1-\eta}{(n-1)s^2}(H+nh)^2 \qquad\text{for every unit }v.
\end{aligned}
\end{equation}
In particular, taking $v=u$ gives
\[
 \frac{(1-\eta)(H+2nh)^2}{(n-1)(1+\eta/8)^2s^2} \le\frac{(1-\eta)H^2}{(n-1)(1+\eta/8)^2s^2} -\Phi_u^{1,\eta}(\Theta_{(1+\eta/8)s}^{-1}(A)) \le\frac{1-\eta}{(n-1)s^2}(H+nh)^2.
\]
It follows that $H+2nh\le(1+\eta/8)(H+nh)$, so
\begin{equation}\label{eq:shifted-h-ratio}
 \frac{nh}{H}\le\frac{\eta}{8-\eta}\le\frac1{15}.
\end{equation}
Thus $H>h$, which proves (i). Using \eqref{eq:shifted-cone-bound} and \eqref{eq:shifted-h-ratio}, we obtain, for every unit vector $v$,
\[
\begin{aligned}
 \frac{\Phi_v^{1,\eta/2}(\Theta_{(1+\eta/8)s}^{-1}(A))}{H^2} &\ge\frac1{(n-1)(1+\eta/8)^2s^2}\left(1-\frac\eta2 -(1-\eta)(1+\eta/8)^2\left(1+\frac{nh}{H}\right)^2\right)\\
 &\ge\frac1{(n-1)(1+\eta/8)^2s^2}\left(1-\frac\eta2 -(1-\eta)\frac{(8+\eta)^2}{(8-\eta)^2}\right)\\
 &=\frac{\eta^2(48+\eta)}{2(n-1)(1+\eta/8)^2s^2(8-\eta)^2}>0.
\end{aligned}
\]
This proves (ii). For $i=2$, replace $\Phi_u^{1,\eta}$ and $\Phi_u^{1,\eta/2}$ by $\frac{n-2}{2(n+1)}\Phi_u^{2,\eta}$ and $\frac{n-2}{2(n+1)}\Phi_u^{2,\eta/2}$, respectively. The same calculation proves (i), (ii), and \eqref{eq:shifted-h-ratio}.
\end{proof}

\begin{lemma}\label{lem:improvement}
Let $Y\subset\{H>0\}$ be invariant under the natural $O(V)\times O(N)$ action. Suppose that $A\in Y$ is preserved by mean curvature flow of closed immersions, and satisfies
\[
 Y\subset\left\{A: A+\frac{h}{H}\Id\otimes\bH\in\cC_{i,\eta}(s)\right\}
\]
for some $s\ge1$ and $h>0$. Then the set
\[
 Y'=\left\{A\in Y: A+\frac{2h}{H}\Id\otimes\bH \in\cC_{i,\eta}\bigl((1+\eta/8)s\bigr)\right\}
\]
is invariant under the natural $O(V)\times O(N)$ action, and $A\in Y'$ is preserved by mean curvature flow of closed immersions. Moreover, we have
\[
 \{A\in Y:H\le h\}\subset Y'.
\]
\end{lemma}
\begin{proof}
Lemma~\ref{lem:family} implies that $Y'$ is invariant under the natural $O(V)\times O(N)$ action, and Lemma~\ref{lem:shifted-estimate}(i) gives $\{A\in Y:H\le h\}\subset Y'$. It remains to prove preservation under mean curvature flow. Suppose first that $i=1$, and define
\[
\begin{aligned}
 \Psi_u(A) &=\frac{\Phi_u^{1,\eta}(\Theta_{(1+\eta/8)s}^{-1}(A))}{H}
 +\frac{4n(1-\eta)h}{(n-1)(1+\eta/8)^2s^2} \left(1+\frac{nh}{H}\right),\qquad |u|=1.
\end{aligned}
\]
By \eqref{eq:shifted-defining-functions}, $Y'=\{A\in Y:\min_{|u|=1}\Psi_u(A)\ge0\}$.

Let $A\in Y$, let $u$ be unit such that $\Psi_u(A)=\min_{|v|=1}\Psi_v(A)\le0$, and let $Z\in\Sym^3(V)\otimes N$. Fix an orthonormal normal frame $\nu_1=\bH/H,\nu_2,\ldots,\nu_q$. Lemma~\ref{lem:shifted-estimate}(ii) gives $A\in\cC_{1,\eta/2}\bigl((1+\eta/8)s\bigr)$, and
\[
 \Phi_u^{1,\eta}(\Theta_{(1+\eta/8)s}^{-1}(A)) =\min_{|v|=1}\Phi_v^{1,\eta}(\Theta_{(1+\eta/8)s}^{-1}(A))<0.
\]
Using the assumption on $Y$, the inequality $\Psi_u(A)\le0$, and \eqref{eq:deformed-splitting}, we obtain
\[
\begin{aligned}
 \frac{n(1-\eta)h(2H+3nh)}{(n-1)(1+\eta/8)^2s^2}
 &\le\frac{\Phi_u^{1,\eta}(\Theta_s^{-1}(A))}{(1+\eta/8)^2} -\Phi_u^{1,\eta}(\Theta_{(1+\eta/8)s}^{-1}(A))\\
 &=\frac{\eta(16+\eta)}{(8+\eta)^2}[-\Phi_u^{1,\eta}(\widehat A)].
\end{aligned}
\]
Applying \eqref{eq:coarse} to $\Theta_{(1+\eta/8)s}^{-1}(A)$ gives
\[
 \frac{|A^{\bH}|^2}{H^2}\le|A|^2 \le(1+\eta/8)^2s^2|\Theta_{(1+\eta/8)s}^{-1}(A)|^2 \le\frac{2H^2}{n}.
\]
Since $D(H^{-1})(A)[\Q(A)]=-|A^{\bH}|^2/H^3$ and \eqref{eq:shifted-h-ratio} gives $(8+\eta)^2/(16+\eta)\ge4+\eta/2\ge4+2nh/H$, Proposition~\ref{prop:deformed-reaction} implies
\[
\begin{aligned}
 &D\Psi_u(A)[\Q(A)]-|A|^2\Psi_u(A)\\
 &\ge\frac{\eta H}{n}[-\Phi_u^{1,\eta}(\widehat A)] -\frac{8(1-\eta)hH(H+2nh)}{(n-1)(1+\eta/8)^2s^2}\\
 &\ge\frac{(1-\eta)hH}{(n-1)(1+\eta/8)^2s^2} \left[\frac{(8+\eta)^2}{16+\eta}(2H+3nh)-8(H+2nh)\right]\\
 &\ge\frac{(1-\eta)hH}{(n-1)(1+\eta/8)^2s^2} \left[\left(4+\frac{2nh}{H}\right)(2H+3nh)-8(H+2nh)\right]\\
 &=\frac{6n^2(1-\eta)h^3}{(n-1)(1+\eta/8)^2s^2}\ge0.
\end{aligned}
\]
Moreover,
\[
 \sum_jD^2(H^{-1})(A)[Z_j,Z_j] =\frac1{H^3}\left(2|\tr Z_1|^2 -\sum_{\alpha>1}|\tr Z_\alpha|^2\right),
\]
so Proposition~\ref{prop:deformed-hessian} and \eqref{eq:shifted-h-ratio} give
\[
 -\sum_jD^2\Psi_u(A)[Z_j,Z_j] \ge\frac1{(n-1)(1+\eta/8)^2s^2H}\left(\frac14 -8(1-\eta)\frac{n^2h^2}{H^2}\right)|\tr Z_1|^2\ge0.
\]
Consequently,
\[
 D\Psi_u(A)[\Q(A)]-\sum_jD^2\Psi_u(A)[Z_j,Z_j] \ge|A|^2\Psi_u(A).
\]
The functions $\Psi_u$ are smooth on $\{H>0\}$, depend continuously on $u$, and satisfy $\Psi_{gu}(gA)=\Psi_u(A)$. Proposition~\ref{prop:maximum} therefore proves preservation with $c=1$. For $i=2$, the same proof applies to $\frac{n-2}{2(n+1)}\Phi_u^{2,\eta}$ and $\frac{n-2}{2(n+1)}\Phi_u^{2,\eta/2}$, using Proposition~\ref{prop:deformed-hessian} for $\frac{n-2}{2(n+1)}\Phi_u^{2,\eta}$.
\end{proof}

\subsection{Pinching sets}
We define
\[
 \cF_{i,\eta}=\bigcap_{s\ge1}\cC_{i,\eta}(s).
\]
If $A\in\cF_{i,\eta}$, then $A\in\cC_{i,\eta}(1)=\cC_i(\eta)\cap\{H>0\}$. By Lemma~\ref{lem:family}, $s^2|\widehat A|^2\le2H^2/n$ for every $s\ge1$, and hence $\widehat A=0$. Conversely, if $A\in\cC_i(\eta)$, $H>0$, and $\widehat A=0$, then $\Theta_s^{-1}(A)=A/s\in\cC_i(\eta)$ for every $s\ge1$. Thus
\[
 \cF_{i,\eta}=\{A\in\cC_i(\eta):H>0,\ \widehat A=0\}.
\]
This set is closed in $\{H>0\}$ and invariant under positive scalar multiplication and the natural $O(V)\times O(N)$ action. Every $A\in\cF_{i,\eta}$ has the form $A=(A^{\bH}/H^2)\otimes\bH$, so its normal image is one-dimensional. Moreover, Cauchy--Schwarz and \eqref{eq:coarse} give
\[
 H^4=(\tr A^{\bH})^2 \le\operatorname{rank}(A^{\bH})|A^{\bH}|^2 \le\frac2n\operatorname{rank}(A^{\bH})H^4.
\]
Consequently, $\operatorname{rank}A^{\bH}\ge n/2\ge2$, and $\cF_{i,\eta}$ satisfies the assumptions in Definition~\ref{def:pinching}.

\begin{proposition}\label{prop:pinching-construction}
Let $\mathcal K$ be a compact subset of $\cC_i(\eta)\cap\{H>0\}$. There exists a $\cF_{i,\eta}$-pinching set $\mathcal P$ such that $\mathcal K\subset\mathcal P\subset\cC_i(\eta)\cap\{H>0\}$.
\end{proposition}
\begin{proof}
Choose $h>0$ such that $H\le h$ on $\mathcal K$. Put $\mathcal P_0=\cC_i(\eta)\cap\{H>0\}$ and define inductively
\[
 \mathcal P_{j+1}=\left\{A\in\mathcal P_j: A+\frac{2^{j+1}h}{H}\Id\otimes\bH \in\cC_{i,\eta}\bigl((1+\eta/8)^{j+1}\bigr)\right\}
\]
for $j\ge0$. Corollary~\ref{cor:preserved} and Proposition~\ref{prop:strict}, applied to each connected component, show that $\mathcal P_0$ is preserved by mean curvature flow. Moreover, \eqref{eq:shifted-defining-functions} gives
\[
 A+\frac{h}{H}\Id\otimes\bH\in\cC_{i,\eta}(1) \qquad(A\in\mathcal P_0).
\]
For $j\ge1$, we have by definition
\[
 A+\frac{2^j h}{H}\Id\otimes\bH \in\cC_{i,\eta}\bigl((1+\eta/8)^j\bigr) \qquad(A\in\mathcal P_j).
\]
Applying Lemma~\ref{lem:improvement} with $Y=\mathcal P_j$, $s=(1+\eta/8)^j$, and $2^j h$ in place of $h$, we conclude by induction that every $\mathcal P_j$ is invariant under the natural $O(V)\times O(N)$ action and preserved by mean curvature flow. The same lemma gives
\[
 \{A\in\mathcal P_j:H\le2^j h\}\subset\mathcal P_{j+1}.
\]
Since $H\le h$ on $\mathcal K\subset\mathcal P_0$, it follows that $\mathcal K\subset\mathcal P_j$ for every $j\ge0$. Therefore $\mathcal P=\bigcap_{j\ge0}\mathcal P_j$ is invariant under the natural $O(V)\times O(N)$ action, preserved by mean curvature flow, and contains $\mathcal K$.

Fix $\rho>0$ and choose an integer $j\ge1$ such that $8/(1+\eta/8)^{2j}\le\rho^2$. Let $A\in\mathcal P$ with $H\ge2^jnh$. Since $A\in\cC_i(\eta)$ and $\widehat A$ is traceless, \eqref{eq:deformed-splitting} and Lemma~\ref{lem:scalar} give
\[
 \Phi_u^{i,\eta}\left(\frac{A^{\bH}}{H^2}\otimes\bH\right) =\Phi_u^{i,\eta}(A)-\Phi_u^{i,\eta}(\widehat A) \ge\Phi_u^{i,\eta}(A)\ge0
\]
for every unit $u$. It follows that $(A^{\bH}/H^2)\otimes\bH\in\cC_i(\eta)$. This tensor has trace $\bH$ and its component perpendicular to $\bH$ vanishes, so $(A^{\bH}/H^2)\otimes\bH\in\cF_{i,\eta}$. Since $A\in\mathcal P_j$, we have
\[
 A+\frac{2^j h}{H}\Id\otimes\bH \in\cC_{i,\eta}\bigl((1+\eta/8)^j\bigr).
\]
Applying Lemma~\ref{lem:family} to this tensor and using $H^2\le n|A|^2$, we obtain
\[
\begin{aligned}
 \operatorname{dist}(A,\cF_{i,\eta})^2 &\le\left|A-\frac{A^{\bH}}{H^2}\otimes\bH\right|^2 =|\widehat A|^2\\
 &\le\frac{2(H+2^jnh)^2}{n(1+\eta/8)^{2j}} \le\frac{8H^2}{n(1+\eta/8)^{2j}} \le\rho^2|A|^2.
\end{aligned}
\]
Consequently, every $A\in\mathcal P$ with $\operatorname{dist}(A,\cF_{i,\eta})>\rho|A|$ satisfies $H<2^jnh$ and, by \eqref{eq:coarse}, $|A|<\sqrt{2n}\,2^j h$. This verifies the last condition in Definition~\ref{def:pinching}.
\end{proof}

\section{Round extinction under the first pinching condition}

Let $F:M\times[0,T)\to\R^{n+q}$ be the maximal smooth mean curvature flow of a closed connected immersion satisfying $\cC_1$ initially. By \cite[Propositions~3.10--3.11]{Smoczyk}, the maximal existence time $T$ is finite and the curvature is unbounded as $t\uparrow T$. Moreover, Proposition~\ref{prop:strict} gives $H>0$ for $t>0$. Fix $\tau=T/4$. By Corollary~\ref{cor:margin}, there exists $\eta\in(0,1/2]$ such that $A\in\cC_1(\eta)$ for $t\ge\tau$. The set
\[
 \mathcal K=\{A(p,\tau):p\in M\} \subset\cC_1(\eta)\cap\{H>0\}
\]
is compact. Proposition~\ref{prop:pinching-construction} therefore gives an $\cF_{1,\eta}$-pinching set $\mathcal P$ containing $\mathcal K$. Since $\mathcal P$ is preserved by mean curvature flow, $A(p,t)\in\mathcal P$ for all $p\in M$ and $\tau\le t<T$.

\begin{proof}[Proof of Theorem~\ref{thm:mainA}]
Assertions (i) and (ii) follow from Corollary~\ref{cor:preserved} and Proposition~\ref{prop:strict}. We prove (iii).

Choose $t_j\uparrow T$ and $p_j\in M$ with
\[
 \lambda_j=|A|(p_j,t_j) =\max_{M\times[0,t_j]}|A|\longrightarrow\infty.
\]
For the rescaled flows defined by \eqref{eq:rescaling}, we have
\[
 |A_j|\le1\quad\text{on }M\times[-\lambda_j^2t_j,0], \qquad |A_j|(p_j,0)=1.
\]
Since $\lambda_j^2t_j\to\infty$, \cite[Corollary~4.4]{NaffCanonical} and a diagonal subsequence give smooth pointed convergence to a connected complete ancient flow. The estimates in Proposition~\ref{prop:derivatives}, applied with initial time $\tau$, and \eqref{eq:coarse} allow us to apply Proposition~\ref{prop:blowup}. Thus the limit lies in a fixed affine $(n+1)$-plane, with $H_\infty>0$ and $A_\infty\in\cF_{1,\eta}$. By \eqref{eq:shape-square-bound}, its principal curvatures satisfy
\[
 \kappa_i\ge\frac{H_\infty}{n} -\frac{\sqrt{1-\eta}}nH_\infty =\frac{1-\sqrt{1-\eta}}nH_\infty>0.
\]
It follows from \cite[Section~2]{NaffModels} that every time slice is embedded and bounds a convex set. The same curvature bound and \cite[Main Theorem~1.1]{HamiltonPinched} imply that each time slice is compact. The compact limit extends to its first singular time, preserving the same pinching. Applying \cite[Theorem~3.2]{HuiskenSinestrariAncient} after translating that singular time to zero shows that the limit is a family of shrinking spheres. Since $|A_\infty|(p_\infty,0)=1$, the sphere at time zero has radius $\sqrt n$.

For large $j$, the embeddings defining pointed convergence are defined on all of $S^n$. Their images are open and closed in the connected manifold $M$, so they are diffeomorphisms $\phi_j:S^n\to M$. We may choose $y_j\in\R^{n+q}$ and a linear isometric embedding $\iota:\R^{n+1}\to\R^{n+q}$ such that
\[
 \lambda_j\bigl(F(\phi_j(\omega),t_j+\lambda_j^{-2}s)-y_j\bigr) \longrightarrow\sqrt{n-2ns}\,\iota(\omega) \qquad(s\le0),
\]
smoothly on compact time intervals.

Let $r_j=\max_M|F(\cdot,t_j)-y_j|$. Then $\lambda_jr_j\to\sqrt n$, and the maximum principle \cite[proof of Proposition~3.10]{Smoczyk} gives
\[
 \max_M|F(\cdot,t)-y_j|^2\le r_j^2-2n(t-t_j) \qquad(t_j\le t<T).
\]
For $k\ge j$, both closed balls $\overline B(y_j,r_j)$ and $\overline B(y_k,r_k)$ contain $F(M,t_k)$, so $|y_j-y_k|\le r_j+r_k$. Since $r_j\to0$, there is a point $x_*$ such that
\[
 y_j\to x_*,\qquad \max_M|F(\cdot,t)-x_*|\longrightarrow0\quad\text{as }t\uparrow T.
\]
Letting $t\uparrow T$ in the preceding inequality gives
\[
 \lambda_j|y_j-x_*|\le\lambda_jr_j,\qquad 0\le\lambda_j^2(T-t_j)\le\frac{(\lambda_jr_j)^2}{2n}.
\]
Passing to a subsequence, we may assume that $\lambda_j(y_j-x_*)\to z$ and $\lambda_j^2(T-t_j)\to a\in[0,1/2]$. Consequently,
\[
 \lambda_j\bigl(F(M,T+\lambda_j^{-2}s)-x_*\bigr) \longrightarrow z+\sqrt{n-2n(s+a)}\,\iota(S^n) \qquad(-2\le s\le-1),
\]
smoothly with multiplicity one. At each of these times, the Gaussian integrals centered at $(0,0)$ converge to the Gaussian density of $F$ at $(x_*,T)$. Thus the Gaussian integral of the limiting family is constant. The equality case of Huisken's monotonicity formula \cite[Proposition~3.16]{Smoczyk} gives $z=0$ and $a=1/2$. We conclude that
\begin{equation}\label{eq:spherical-sequence}
 \lambda_j\bigl(F(M,T+\lambda_j^{-2}s)-x_*\bigr) \longrightarrow\sqrt{-2ns}\,\iota(S^n) \qquad(-2\le s\le-1).
\end{equation}

The measures induced by $F$, counted with multiplicity, form an integral Brakke flow by \cite[Remark~1.2 and Definition~A.1]{LSS}. The convergence in \eqref{eq:spherical-sequence} verifies the measure-closeness assumption of \cite[Theorem~1.1]{Schulze}, while its density assumption holds with equality by Huisken's monotonicity formula. Applying that theorem to one sufficiently large fixed rescaling gives smooth convergence of the normalized images, as normal graphs of unit density, to a self-shrinker. By \eqref{eq:spherical-sequence}, the self-shrinker is $\sqrt{2n}\,\iota(S^n)$. For $t$ sufficiently close to $T$, the immersion $F(\cdot,t)$ is a proper local diffeomorphism onto its image and hence a covering. Since the density is one, this covering is one-sheeted. Composing its inverse with the normal graph parametrization gives smooth convergence after reparametrization.
\end{proof}

\section{Canonical neighborhoods under the second pinching condition}\label{sec:canonical}

Let $F:M\times[0,T)\to\R^{n+q}$ be the maximal smooth mean curvature flow of a closed connected immersion satisfying $\cC_2$ initially. As in Section~6, the maximal existence time $T$ is finite, the curvature is unbounded as $t\uparrow T$, and $H>0$ for $t>0$. Fix $\tau=T/4$. By Corollary~\ref{cor:margin}, there exists $\eta\in(0,1/2]$ such that $A\in\cC_2(\eta)$ for $t\ge\tau$. The set
\[
 \mathcal K=\{A(p,\tau):p\in M\} \subset\cC_2(\eta)\cap\{H>0\}
\]
is compact. Proposition~\ref{prop:pinching-construction} gives an $\cF_{2,\eta}$-pinching set $\mathcal P$ containing $\mathcal K$. Since $\mathcal P$ is preserved by mean curvature flow, $A(p,t)\in\mathcal P$ for all $p\in M$ and $\tau\le t<T$. Throughout this section, we apply Proposition~\ref{prop:derivatives} and Corollary~\ref{cor:local} with initial time $\tau$.

\begin{lemma}\label{lem:cylindrical-limit}
Along the original flow,
\[
 \limsup_{H\to\infty}\frac{|A|^2}{H^2}\le\frac1{n-1}.
\]
\end{lemma}
\begin{proof}
Since $|A|$ is unbounded, \eqref{eq:coarse} implies that $H$ is unbounded. Let $\beta=\limsup_{H\to\infty}|A|^2/H^2\in[1/n,2/n]$, and choose $(p_j,t_j)$ such that
\[
 Q_j=H(p_j,t_j)\longrightarrow\infty,\qquad \frac{|A|^2}{H^2}(p_j,t_j)\longrightarrow\beta.
\]
The curvature is bounded on each compact subinterval of $[0,T)$, so $t_j\to T$. For the rescaled flows \eqref{eq:rescaling} with $\lambda_j=Q_j$, Corollary~\ref{cor:local} and \eqref{eq:coarse} give
\[
 \tfrac12\le H_j\le2,\qquad |A_j|^2\le\frac8n \quad\text{on }B_{g_j(0)}(p_j,r)\times[-\theta,0],
\]
where $r,\theta>0$ are independent of $j$. After decreasing $r$ and $\theta$, we obtain a smooth pointed limit near $(p_\infty,0)$ by \cite[Corollary~4.4]{NaffCanonical}. Proposition~\ref{prop:blowup} implies that the limit is a hypersurface flow with $H_\infty>0$ and $A_\infty\in\cF_{2,\eta}$. Since $H\ge Q_j/2\to\infty$ on the corresponding neighborhoods of the original flow, the definition of $\beta$ gives
\[
 \frac{|A_\infty|^2}{H_\infty^2}\le\beta,\qquad \frac{|A_\infty|^2}{H_\infty^2}(p_\infty,0)=\beta.
\]
By Lemma~\ref{lem:ratio}, the nonzero principal curvatures are equal and positive. Suppose that two principal curvatures vanish, and choose corresponding orthonormal eigenvectors $e_1,e_2$. Then $|A_\infty^\circ|^2\ge2H_\infty^2/[n(n-2)]$, so \eqref{eq:phi2} gives
\[
\begin{aligned}
 0\le\frac{\Phi_{e_1\wedge e_2}^{2,\eta}(A_\infty)}{H_\infty^2} &\le\frac{2(n+1)(1-\eta)}{(n-1)(n-2)} -\frac{4(2n-1)}{n(n-1)(n-2)}-\frac2n\\
 &=-\frac{2(n+1)\eta}{(n-1)(n-2)}<0,
\end{aligned}
\]
a contradiction. Hence at least $n-1$ principal curvatures are nonzero, and $\beta\le1/(n-1)$.
\end{proof}

By Proposition~\ref{prop:blowup} and Lemma~\ref{lem:cylindrical-limit}, every connected nonflat smooth limit of rescalings with scales tending to infinity lies in a fixed affine $(n+1)$-plane and satisfies
\begin{equation}\label{eq:cylindrical-limit}
 H_\infty>0,\qquad |A_\infty|^2\le\frac{H_\infty^2}{n-1}.
\end{equation}
Let $\lambda_1\le\cdots\le\lambda_n$ denote its principal curvatures in the mean-curvature direction. Cauchy--Schwarz gives
\[
\begin{aligned}
 (H_\infty-\lambda_1)^2 &\le(n-1)(|A_\infty|^2-\lambda_1^2) \le H_\infty^2-(n-1)\lambda_1^2,\\
 \lambda_1+\lambda_2 &\ge H_\infty-\sqrt{(n-2)|A_\infty|^2} \ge\left(1-\sqrt{\frac{n-2}{n-1}}\right)H_\infty.
\end{aligned}
\]
The first inequality gives $n\lambda_1^2\le2H_\infty\lambda_1$ and hence $\lambda_1\ge0$. Consequently, the limit is weakly convex and uniformly two-convex. The derivative estimates \eqref{eq:derivatives} hold with the same constants $C_1,C_2$. If the limit is complete and ancient, it is noncollapsed by \cite[Theorem~1.1]{NaffModels}. In the strictly convex case, \cite[Theorem~5.3]{NaffModels} gives a noncollapsing constant depending only on $n,C_1,C_2$; otherwise, the reduction in \cite[Section~2]{NaffModels} shows that it is a shrinking round cylinder. Thus all these limits are $\alpha$-noncollapsed for one $\alpha=\alpha(n,C_1,C_2)>0$.

\begin{definition}\label{def:comparison}
An \emph{ancient model solution} is a complete, connected, nonflat ancient mean curvature flow contained in an affine $(n+1)$-plane in $\R^{n+q}$, uniformly two-convex and noncollapsed within that plane, as in \cite[Definition~2.2]{NaffCanonical}. For $H(p,t)>0$ and $r,\theta>0$, set
\[
 \widehat P(p,t,r,\theta)=B_{g(t)}(p,H(p,t)^{-1}r)\times[t-\theta H(p,t)^{-2},t].
\]
The flow is $\varepsilon_0$-close to an ancient model solution on $\widehat P(p_0,t_0,K_0,K_0)$ if its rescaling \eqref{eq:rescaling}, with scale $Q_0=H(p_0,t_0)$ and basepoint $(p_0,t_0)$, satisfies \cite[Definition~2.5]{NaffCanonical}.
\end{definition}

\begin{proof}[Proof of Theorem~\ref{thm:mainB}]
Assertions (i) and (ii) follow from Corollary~\ref{cor:preserved} and Proposition~\ref{prop:strict}. We prove (iii).

We prove the assertion with $\alpha$-noncollapsed ancient model solutions, where $\alpha$ is the constant obtained above. Choose $\widetilde\varepsilon>0$ sufficiently small and $\widetilde K>0$ sufficiently large as in \cite[Proposition~2.3, Remark~2.4, and Section~3]{NaffCanonical}, and fix $\varepsilon_0\in(0,\widetilde\varepsilon)$ and $K_0>\widetilde K$.

Suppose this assertion is not true. By the point-picking argument in \cite[Section~3]{NaffCanonical}, we may choose $(p_j,t_j)$ with $Q_j=H(p_j,t_j)\to\infty$ such that the flow is not $\varepsilon_0$-close on $\widehat P(p_j,t_j,K_0,K_0)$ to an $\alpha$-noncollapsed ancient model solution, but has such a neighborhood whenever $t\le t_j$ and $H(x,t)\ge4Q_j$. For the rescaled flows \eqref{eq:rescaling} with $\lambda_j=Q_j$, we have $t_j\to T$ and $Q_j^2(t_j-\tau)\to\infty$.

Fix $\rho>0$. We claim that
\[
 \limsup_{j\to\infty}\ \sup_{x\in B_{g(t_j)}(p_j,\rho/Q_j)} \frac{H(x,t_j)}{Q_j}<\infty.
\]
Suppose otherwise. We follow the proof of \cite[Proposition~3.5]{NaffCanonical}, taking a local limit of the rescaled flows and then rescaling about points where its curvature becomes unbounded. By passing to a diagonal subsequence, we obtain the same limits from rescalings of the original flow with scales tending to infinity. Proposition~\ref{prop:derivatives} and Corollary~\ref{cor:local} give the derivative and local curvature bounds required in that proof. At points with $H\ge4Q_j$, the choice of $(p_j,t_j)$ and \cite[Proposition~2.3]{NaffCanonical} give a neck, a cap, or a closed neighborhood diffeomorphic to $S^n$. Moreover, Corollary~\ref{cor:local} and \cite[Corollary~4.4]{NaffCanonical} give a smooth backward extension of each nonflat time-zero limit on every relatively compact subset, for a time interval depending on that subset. Such a limit is a weakly convex hypersurface flow by \eqref{eq:cylindrical-limit}, and the Gauss equation gives
\[
 \operatorname{Ric}\ge0,\qquad \frac{n-2}{n-1}H_\infty^2 \le\operatorname{scal}=H_\infty^2-|A_\infty|^2 \le H_\infty^2.
\]
The area comparison and Step~5 in the proof of Proposition~3.5 give rescaled local flows with uniformly bounded curvature on a common backward time interval. Their limit is a nonflat hypersurface flow whose time-zero slice is a piece of a metric cone. The strong maximum principle excludes this limit as in that proof. This proves the claim.

By the proofs of \cite[Lemmas~3.6--3.7]{NaffCanonical}, a subsequence converges smoothly to a complete hypersurface flow on a maximal backward time interval $(t_0,0]$, with bounded curvature on compact time subintervals. Suppose this limit flow is not ancient, so $t_0>-\infty$. The limit is weakly convex and satisfies $|A_\infty|/H_\infty\le1/\sqrt{n-1}<1$. Hence the distance estimate and Harnack inequality in \cite[Lemma~2.6 and Theorem~2.7]{NaffCanonical} apply. As in the proof of \cite[Lemma~3.8]{NaffCanonical}, pass to a smooth pointed limit of time slices with times tending to $t_0$. This limit can be obtained along a diagonal sequence of the original rescaled flows. Using the local backward extensions above, the same argument bounds $\sup_{M_\infty}H_\infty(\cdot,t)$ as $t\downarrow t_0$, contradicting the curvature blow-up in the maximal backward construction. The limit is therefore ancient.

The limit is an $\alpha$-noncollapsed ancient model solution, since $H_\infty(p_\infty,0)=1$ and it inherits \eqref{eq:derivatives}. Smooth pointed convergence on $B_{g_\infty(0)}(p_\infty,K_0+1)\times[-K_0-1,0]$ gives the required $\varepsilon_0$-closeness on $\widehat P(p_j,t_j,K_0,K_0)$ for large $j$. This contradicts the choice of $(p_j,t_j)$.
\end{proof}

\appendix
\section{The examples}\label{app:examples}

Fix $k\ge1$ and $a_0,b_0>0$, and let $T=a_0^2/(2k)+b_0^2/[2(k+1)]$. For $0\le t<T$, define $a(t)>0$ and $b(t)$ by
\[
 \frac{a(t)^2}{2k}+\frac{b_0^2}{2(k+1)}\left(\frac{a(t)}{a_0}\right)^{4(k+1)/k}=T-t,\qquad b(t)=b_0\left(\frac{a(t)}{a_0}\right)^{2(k+1)/k}.
\]
The first equation has a unique positive solution, since its left-hand side is strictly increasing in $a(t)$. In $\R^{k+1}\oplus\Sym_0(\R^{k+1})\cong\R^N$, the immersions $F:S^k\times[0,T)\to\R^N$ given by
\[
 F(x,t)=\left(a(t)x,b(t)\left(xx^T-\frac{\Id}{k+1}\right)\right),\qquad x\in S^k
\]
solve mean curvature flow with initial values $a(0)=a_0$ and $b(0)=b_0$. Let $\rho(t)=\sqrt{a(t)^2+2b(t)^2}$. In the orthonormal frame $dF(e_i)/\rho(t)$, where $e_1,\ldots,e_k$ is an orthonormal basis of $T_xS^k$, we have
\[
\begin{aligned}
 A_{ij}&=\frac1{\rho(t)^2}\left(-a(t)\delta_{ij}x,b(t)(e_ie_j^T+e_je_i^T-2\delta_{ij}xx^T)\right),\\
 A^\circ_{ij}&=\frac{b(t)}{\rho(t)^2}\left(0,e_ie_j^T+e_je_i^T-\frac{2\delta_{ij}}k(\Id-xx^T)\right).
\end{aligned}
\]

\noindent\textit{Example~\ref{ex:intro-sphere}.} The parametrization in the introduction is $F(\cdot,0)$ with $k=n$ and $(a_0,b_0)=(2\sqrt n,1)/(2n+1)$. At $t=0$, the formulas above give
\[
 |A|^2=\frac{3n(n+1)}2,\qquad H^2=n(n^2+n+1),\qquad \sum_\alpha(A_\alpha^\circ)^2=\frac{|A^\circ|^2}{n}\Id=\frac{(n-1)(n+2)}{2n}\Id.
\]
Thus $|A|^2/H^2=3(n+1)/[2(n^2+n+1)]>4/(3n)>1/(n-1)$ for $n\ge10$, whereas
\[
 \frac{H^2}{n-1}\Id-|A^\circ|^2\Id-n\sum_\alpha(A_\alpha^\circ)^2=\frac{n^2+4n-2}{n-1}\Id>0.
\]
Since $b(t)/a(t)\to0$ and $a(t)^2/(T-t)\to2n$, we have $(T-t)^{-1/2}F(\cdot,t)\to(\sqrt{2n}\,x,0)$ smoothly, with multiplicity one.

\noindent\textit{Example~\ref{ex:intro-tube}.} Take $k=9$, $a_0=\sqrt{35}/6$, and $b_0=1/(6\sqrt2)$. For $(x,z)\in S^9\times S^1$, the flow is given by
\[
 F(x,z,t)=\left(a(t)x,b(t)\left(xx^T-\frac{\Id}{10}\right),\sqrt{4-2t}\,z\right).
\]
At $t=0$, $\rho(0)=1$. Complete the frame by a unit tangent to the circle. For $i,j\le9$, $A_{ij}$ is given by the formula for $A_{ij}$ above with zero last component, while $A_{10,10}=(0,0,-z/2)$ and $A_{i,10}=0$. Its traceless part is
\[
\begin{aligned}
 A^\circ_{ij}&=\left(-\frac{a_0\delta_{ij}}{10}x,b_0\left(e_ie_j^T+e_je_i^T-\frac{\delta_{ij}}5\Id\right),\frac{\delta_{ij}}{20}z\right),\\
 A^\circ_{10,10}&=\left(\frac{9a_0}{10}x,2b_0\left(xx^T-\frac{\Id}{10}\right),-\frac9{20}z\right),\qquad A^\circ_{i,10}=0.
\end{aligned}
\]
These give $|A|^2=12$, $H^2=84$, and $|A^\circ|^2=18/5$, so $|A|^2/H^2=1/7>4/30>1/8$. On $\Lambda^2T(S^9\times S^1)$, we have
\[
 \frac{11H^2}{36}\Id-\frac{38}{9}|A^\circ|^2\Id-6\sum_\alpha((A_\alpha^\circ)^2)^{[2]}-2\sum_\alpha((A_\alpha^\circ)^{[2]})^2=6\Id_{\Lambda^2TS^9}\oplus\frac13\Id_{TS^9\wedge TS^1}>0.
\]
Hence the cone inequality is strict. Here $T=709/12960<2$, so the limiting circle has positive radius. Since $b(t)/a(t)\to0$ and $a(t)^2/(T-t)\to18$, parabolic rescalings at each point of this circle converge smoothly locally to the shrinking round cylinder $S^9\times\R$, with multiplicity one.

\renewcommand{\bibliofont}{\footnotesize\setlength{\baselineskip}{10.25pt}}

\end{document}